\documentclass[12pt]{amsart}
\usepackage{amscd}
\usepackage{amssymb}
\usepackage{a4wide}
\usepackage{amstext}
\usepackage{amsthm}
\usepackage{mathrsfs}
\usepackage{relsize}
\usepackage[final]{hyperref}
\hypersetup{unicode= false, colorlinks=true, linkcolor=blue,
anchorcolor=blue, citecolor=green, filecolor=red, menucolor=blue,
pagecolor=blue, urlcolor=blue} 
\numberwithin{equation}{section}

\newtheorem{Thm}{Theorem}[section]
\newtheorem{theorem}[Thm]{Theorem}
\newtheorem{lemma}[Thm]{Lemma}
\newtheorem{proposition}[Thm]{Proposition}
\newtheorem{corollary}[Thm]{Corollary}
\newtheorem{remark}[Thm]{Remark}

\newtheorem{question}[Thm]{Question}

\DeclareMathOperator{\Ai}{Ai}
\DeclareMathOperator{\Bi}{Bi}

\renewcommand{\Re}{\operatorname{Re}}
\renewcommand{\Im}{\operatorname{Im}}

\begin{document}

\title[Second-order differential equations and Cauchy kernels]
{Second-Order Differential Equations and Sums of Squares of Cauchy Kernels
with Finitely Many Zeros}

\author{Vladimir Shemyakov}
\address{St. Petersburg Department of V. A. Steklov Institute of Mathematics,
Russian Academy of Sciences, St. Petersburg, Russia}
\email{vladimir.v.shemyakov@gmail.com}


\begin{abstract}
We study finite-order meromorphic functions representable as absolutely
convergent sums of squares of Cauchy kernels and having only finitely many
zeros. By earlier work of Baranov and the author, such functions admit a
representation $f=P/g^2$, where $P$ is a polynomial and $g$ is entire,
satisfying the differential equation
$
Pg''-P'g'+Qg=0,
$
where $Q$ is a polynomial. We show that the zeros of $g$ asymptotically
accumulate along the Stokes rays. If $\mathrm{deg}\ Q>\mathrm{deg}\ P$, they approach these
rays in the Euclidean metric, whereas in the borderline case
$\mathrm{deg}\ Q=\mathrm{deg}\ P$ one obtains in general only localization in logarithmic
neighborhoods of the Stokes rays, and this is sharp. We then characterize the
existence of a decomposition
$
P/g^2=\sum c_n (z-t_n)^{-2}
$
in terms of the sectorial behavior of $g$ and, equivalently, in terms of the
Laine condition for the corresponding Schwarzian equation. Finally, for fixed
$P$ and fixed order, we identify the resulting families, modulo the natural
equivalence relation, with finite-dimensional affine algebraic varieties.
\end{abstract}

\maketitle

\section{Introduction and main results}

Consider meromorphic functions of the form
\begin{equation}\label{eq:intro-cauchy}
f(z)=\sum_{n=1}^{\infty}\frac{c_n}{(z-t_n)^m},
\qquad m\in\mathbb N,
\end{equation}
where
\begin{equation}\label{eq:intro-cond}
|t_n|\to\infty \quad (n\to\infty),
\qquad
\sum_{n=1}^{\infty}\frac{|c_n|}{|t_n|^m}<\infty.
\end{equation}
Under \eqref{eq:intro-cond}, the series in \eqref{eq:intro-cauchy}
converges locally uniformly in
$\mathbb C\setminus\{t_n:n\in\mathbb N\}$, and hence defines a meromorphic
function. We are concerned with the zero distribution of such functions, with
special emphasis on the borderline case $m=2$.

The case $m=1$ is classical. Keldysh proved that, under the assumptions
\begin{equation}\label{eq:intro-keldysh-extra}
\sum_{n=1}^{\infty}|c_n|<\infty,
\qquad
\sum_{n=1}^{\infty}c_n\neq0,
\end{equation}
and provided that the sequence of poles has finite lower order, the function
\begin{equation}\label{eq:sum_cauchy_kernels}
f(z)=\sum_{n=1}^{\infty}\frac{c_n}{z-t_n}
\end{equation}
has infinitely many zeros; see \cite{Keldysh} and
\cite[Ch.~V, Thm.~6.2]{Book_Gold_Ostr}. Clunie, Eremenko, and Rossi
\cite{ClunieEremenkoRossi1993} conjectured that the same conclusion holds
when $c_n>0$ for all $n$ and
\begin{equation}\label{eqref:cond_c_n_t_n}
\sum_{n=1}^{\infty}\frac{|c_n|}{|t_n|}<\infty.
\end{equation}
For further results on sums of simple Cauchy kernels, see
\cite{ClunieEremenkoRossi1993,Eremenko_Langley_Rossi_Sum_Cauchy_Kernels}.

For $m\ge3$, Langley and Rossi proved that every finite-order function of the
form \eqref{eq:intro-cauchy} has infinitely many zeros and that the deficiency
of the value $0$ is strictly less than one, although it need not be zero;
see \cite[Thm.~1.9]{LangleRossi}. For $m=2$, the same authors proved a similar result, but under the additional
assumption \eqref{eqref:cond_c_n_t_n}. In \cite[Thm.~1.1]{baranov2023zeros}, Baranov and the
author established for $m=2$ an analog of Keldysh's theorem under
\eqref{eq:intro-cond} and \eqref{eq:intro-keldysh-extra}, with deficiency zero.

In the quadratic case $m=2$, there are numerous elementary examples. The
identity
\begin{equation}\label{eq:intro-sine-square-zero-free}
\frac{\pi^2}{\sin^2\pi z}
=
\sum_{n\in\mathbb Z}\frac{1}{(z-n)^2}
\end{equation}
gives a zero-free example. On the other hand, changing only the signs of the
Cauchy coefficients gives
\begin{equation}\label{eq:intro-sine-square-alternating}
\frac{\pi^2\cos\pi z}{\sin^2\pi z}
=
\sum_{n\in\mathbb Z}\frac{(-1)^n}{(z-n)^2},
\end{equation}
which has infinitely many zeros. Thus, the finite-zero condition in the
quadratic case is sensitive not only to the pole set, but also to the
asymptotic behaviour of the coefficients.

Starting point of present research is observation, made in \cite[Prop.~1.2]{baranov2023zeros}: if a function of the
form \eqref{eq:intro-cauchy} with $m=2$ has only finitely many zeros, then
it can be written as
\[
f=\frac{P}{g^2},
\]
where $P$ is a polynomial, $g$ is entire, and $g$ satisfies
\begin{equation}\label{eq:intro-main-ode}
Pg''-P'g'+Qg=0
\end{equation}
with an entire coefficient $Q$. In the finite-order case, $Q$ is in fact a
polynomial with $\deg Q\ge \deg P$; see Theorem~\ref{thm:finite-zero-reduction} below. Thus, the original
meromorphic problem leads naturally to the zero distribution of entire
solutions of second-order linear differential equations with rational
coefficients.

Write
\[
p=\deg P,\qquad q=\deg Q,\qquad N=q-p\ge0
\]
in the finite-order case. Classical asymptotic theory due to Hille
\cite{Hille_ODE_complex} shows that the zeros of entire solutions of
\eqref{eq:intro-main-ode} are organized by the corresponding Stokes geometry.
In the model case $P\equiv1$, Euclidean localization of zeros near the
Stokes rays is classical; see, for instance, Bank
\cite{BankZerosPolynomial}, Hellerstein--Rossi
\cite{HellersteinRossi1989}, and the recent expository paper of
Gundersen--Heittokangas--Zemirni
\cite{Gunder_Heittokangas_Asymp_Int_Theory}. Our first main result shows that
the same Euclidean localization remains valid for the general equation
\eqref{eq:intro-main-ode} whenever $N>0$. By contrast, in the borderline case
$N=0$, that is, $\deg Q=\deg P$, one can guarantee in general only
localization in logarithmic neighborhoods of the Stokes rays, and this
description is sharp; see Theorem~\ref{thm:zeros-stokes-localization}.

Our second main result returns from the differential equation to the original
representation problem. It shows that, for an entire solution $g$ of
\eqref{eq:intro-main-ode} with no common zeros with $P$, the
representability of $P/g^2$ in the form
\[
\frac{P(z)}{g^2(z)}
=
\sum_{n=1}^{\infty}\frac{c_n}{(z-t_n)^2}
\]
is equivalent to a sectorial growth condition on $g$, and also to the
occurrence of zeros near all Stokes rays; see
Theorems~\ref{thm:razloz_v_sum_1} and
\ref{thm:sectorial-behavior}. In the Schwarzian formulation, the same condition becomes equivalent to the
Laine condition for the rational Schwarzian term associated with
\eqref{eq:intro-main-ode}; see
Theorems~\ref{thm:laine-entire-solutions} and
\ref{thm:decomposition-from-laine}.

The preceding results give a more precise form of rigidity both of poles $t_n$ and coefficients $c_n$ in finite zero case. The poles of $P/g^2$, that is, the zeros of $g$, are
asymptotically distributed along the Stokes rays, and the coefficients
have a compatible asymptotic behavior. More precisely,
Theorem~\ref{thm:zeros-stokes-localization} and
Corollary~\ref{cor:coefficient-summability} give, on the $k$-th Stokes
sequence,
\[
t_{k,n}\sim B_k\,n^{\frac{2}{N+2}},
\qquad
c_{k,n}\sim A_k\,n^{-\frac{N}{N+2}},
\qquad
A_kB_k\neq0.
\]
In particular, along each fixed Stokes sequence the arguments of the
coefficients $c_{k,n}$ are asymptotically constant, and
\[
\sum_{k,n}\frac{|c_{k,n}|}{|t_{k,n}|^s}<\infty
\quad\Longleftrightarrow\quad
s>1.
\]
This explains the contrast between
\eqref{eq:intro-sine-square-zero-free} and
\eqref{eq:intro-sine-square-alternating}. In both examples $N=0$, so the two
Stokes directions are the positive and negative real axes. In
\eqref{eq:intro-sine-square-zero-free} the Cauchy coefficients are constant
on each ray, whereas in \eqref{eq:intro-sine-square-alternating} they alternate
in sign along each ray. Thus, the latter example violates the asymptotic
constancy of the coefficient arguments forced in the finite-zero case.

Our third main result is algebraic. Once one fixes the polynomial $P$ and
the order $\rho$, one may ask for the size of the family of meromorphic
functions of the form \eqref{eq:intro-cauchy} with $m=2$ and prescribed
finite zero divisor $P$. We show that the corresponding equivalence classes
are parametrized by a complex affine algebraic variety of dimension
$2\rho-1$; more precisely, this quotient is naturally identified with the
corresponding Laine variety; see Theorems~\ref{thm:laine-variety} and
\ref{thm:moduli-bijection}.

The paper is organized as follows. Section~2 contains the reduction to a second-order differential equation,
together with the growth facts for meromorphic functions of the form
\eqref{eq:intro-cauchy}. It also establishes the asymptotic distribution of
zeros of entire solutions of the resulting equation.
In Section~3 we reformulate the problem in terms of the Schwarzian derivative
and establish the criterion for representability by sums of squares of Cauchy
kernels. Section~4 contains the algebraic parametrization of the corresponding
classes. Section~5 contains explicit examples.

\textbf{Acknowledgement.} The author is grateful to Anton Baranov for valuable comments and insightful remarks.

\section{Differential-equation reduction}

In this section we derive the reduction from meromorphic functions represented
by sums of squares of Cauchy kernels to second-order linear differential
equations with rational coefficients. We use the standard notation of Nevanlinna theory, in particular
$T(r,f),N(r,f),m(r,f),n(r,f)$, the deficiencies
$\delta(a,f),\Delta(a,f)$, and the order, lower order, type, and lower type
$\rho(f),\lambda(f),\sigma(f),\tau(f)$; the same symbols are used for the
corresponding characteristics of a sequence $\{t_n\}$, see
\cite{Book_Gold_Ostr,hayman1964meromorphic}. 
\subsection{Growth of Cauchy kernels}

We begin with the classical case of simple Cauchy kernels.

\begin{lemma}[{\cite[Ch.~V, Thm.~6.1]{Book_Gold_Ostr}}]\label{lem:smirnov}
Let $\{c_n\}_{n\in\mathbb N},\{t_n\}_{n\in\mathbb N}\subset\mathbb C$
satisfy
\[
\sum_{n=1}^{\infty}\frac{|c_n|}{|t_n|}<\infty,
\qquad
f(z)=\sum_{n=1}^{\infty}\frac{c_n}{z-t_n}.
\]
Then for every $p\in(0,1)$,
\[
\int\limits_0^{2\pi}|f(re^{i\varphi})|^p\,d\varphi\longrightarrow0,
\qquad r\to\infty.
\]
In particular,
\[
m(r,f)=o(1),
\qquad
T(r,f)=N(r,f)+o(1),
\qquad r\to\infty,
\]
and the order and type of $f$ coincide with those of the pole sequence
$\{t_n\}$.
\end{lemma}

The next statement is the corresponding version for higher powers of the
Cauchy kernel. While our main application is the case $m=2$, we formulate
it for all $m\ge2$ in a way convenient for later use.

\begin{lemma}\label{lem:equality_orders}
Let $m\ge2$, let $\{c_n\}_{n\in\mathbb N},\{t_n\}_{n\in\mathbb N}\subset\mathbb C$, and
assume that
\[
\sum_{n=1}^{\infty}\frac{|c_n|}{|t_n|^m}<\infty,
\qquad
f(z)=\sum_{n=1}^{\infty}\frac{c_n}{(z-t_n)^m}.
\]
Then
\[
T(r,f)\sim N(r,f),
\qquad r\to\infty,
\]
outside a set of finite measure. In particular, the order, lower
order, type, and lower type of $f$ coincide with those of the sequence
$\{t_n\}$.
\end{lemma}

\begin{proof}
Set
\[
g(z)=\sum_{n=1}^{\infty}\frac{c_n/t_n^{m-1}}{z-t_n},
\qquad
H(z)=z^{m-1}g(z).
\]
Since
\[
\sum_{n=1}^{\infty}\frac{|c_n/t_n^{m-1}|}{|t_n|}
=
\sum_{n=1}^{\infty}\frac{|c_n|}{|t_n|^m}<\infty,
\]
Lemma~\ref{lem:smirnov} applies to $g$. Hence $g$, and therefore also
$H$, has the same order and type as the sequence $\{t_n\}$, and
\[
m(r,H)=o\!\bigl(T(r,H)\bigr),
\qquad
T(r,H)\sim N(r,H),
\]
as $r\to\infty$. 
Using
\[
\frac{z^{m-1}}{t^{m-1}(z-t)}
=
\frac{1}{z-t}+\sum_{k=0}^{m-2}\frac{z^k}{t^{k+1}},
\]
we obtain
\[
H^{(m-1)}(z)=(-1)^{m-1}(m-1)!\,f(z).
\]
Moreover, $f$ and $H$ have the same poles, with multiplicities multiplied
by $m$, so $N(r,f)=m\,N(r,H)$.
By the lemma on the logarithmic derivative,
\[
m\!\left(r,\frac{H^{(m-1)}}{H}\right)=O(\log T(r,H))=o(T(r,H))
\]
outside a set of finite measure, see \cite[Ch.~3, Thm.~1.3]{Book_Gold_Ostr}. Therefore
\[
m(r,f)\le
m\!\left(r,\frac{H^{(m-1)}}{H}\right)+m(r,H)+O(1)
=
o\!\bigl(T(r,H)\bigr)
=
o\!\bigl(N(r,f)\bigr),
\]
and so
\[
T(r,f)=N(r,f)\bigl(1+o(1)\bigr)
\]
outside a set of finite measure.

The assertions concerning the order $\rho(f)$, lower order $\lambda(f)$,
type $\sigma(f)$, and lower type $\tau(f)$ are standard. Indeed,
asymptotic equivalence outside a set of finite measure preserves these
characteristics, by the continuation lemma for inequalities; see
\cite[Ch.~6, Lem.~2.1]{Book_Gold_Ostr}.
\end{proof}

\subsection{From $f=P/g^2$ to the associated differential equation}

Assume that a meromorphic function $f$ has only double poles and only
finitely many zeros. Then one may write
\[
f=\frac{P}{g^2},
\]
where $P$ is a polynomial, $g$ is entire, and the zeros of $g$ are
simple and coincide with the poles of $f$. For functions represented by
sums of squares of Cauchy kernels,
\[
f(z)=\sum_{k=1}^\infty \frac{c_k}{(z-t_k)^2},
\qquad
\sum_{k=1}^\infty \frac{|c_k|}{|t_k|^2}<\infty,
\]
this representation is natural, since the poles of $f$ are precisely the
zeros of $g$, and
\[
c_k=\frac{P(t_k)}{g'(t_k)^2}.
\]
As observed in \cite[Prop.~1.2]{baranov2023zeros}, the vanishing of the
residues forces $g$ to satisfy a second-order linear differential equation.
The next theorem adds further information about the solution $g$.

\begin{theorem}\label{thm:finite-zero-reduction}
Let $\{c_n\}_{n\in\mathbb N},\{t_n\}_{n\in\mathbb N}\subset\mathbb C$
satisfy
\begin{equation}\label{eq:cauchy-square-series}
f(z)=\sum_{n=1}^{\infty}\frac{c_n}{(z-t_n)^2},
\qquad
\sum_{n=1}^{\infty}\frac{|c_n|}{|t_n|^2}<\infty.
\end{equation}
Assume that $f$ has only finitely many zeros. Then there exist an entire
function $g$, a polynomial $P\not\equiv0$, and an entire function $Q$
such that
\[
f=\frac{P}{g^2},
\qquad
Pg''-P'g'+Qg=0.
\]
If, in addition, $f$ has finite order, then $Q$ is a polynomial,
\[
\deg Q\ge \deg P,
\qquad
\rho_0:=\frac{\deg Q-\deg P+2}{2}
\in\left\{1,\frac32,2,\frac52,\dots\right\}.
\]
Both $f$ and $g$ have order and lower order equal to $\rho_0$, and both
are of normal type:
\[
0<\tau(f)=\sigma(f)<\infty,
\qquad
0<\tau(g)=\sigma(g)<\infty.
\]
Finally, if
\[
P(z)=\prod_{j=1}^{n}(z-p_j)^{\alpha_j},
\]
then
\[
\frac{Q(z)}{P(z)}
=
S(z)+
\sum_{j=1}^{n}
\alpha_j\,\frac{g'(p_j)}{g(p_j)}\,\frac{1}{z-p_j},
\]
where $S$ is a polynomial.
\end{theorem}

The finite-order part of the proof relies on three auxiliary facts. The first
one excludes order below one for sums of squares of Cauchy kernels with only
finitely many zeros.

\begin{proposition}\label{prop:squares-order}
Let $f$ be as in Theorem~\ref{thm:finite-zero-reduction}. If $f$ has
finite order $\rho(f)$ and only finitely many zeros, then
\[
\rho(f)\ge 1.
\]
\end{proposition}

\begin{proof}
By Lemma~\ref{lem:equality_orders}, the pole sequence $\{t_n\}$ has the
same finite order as $f$. Define
\[
F(z)=\sum_{n=1}^{\infty}\left(\frac{c_n}{z-t_n}+\frac{c_n}{t_n}\right).
\]
Then
\[
F(z)=z\,h(z),
\qquad
h(z)=\sum_{n=1}^{\infty}\frac{c_n/t_n}{z-t_n},
\]
and Lemma~\ref{lem:smirnov} implies that $\rho(h)=\rho(F)=\rho(t_n)$. Since $F'=f$, the case
$\rho(f)<1$ is excluded by a deep theorem of Eremenko, Langley, and Rossi
on zeros of derivatives of meromorphic functions, which implies that $f$
has infinitely many zeros; see \cite[Thm.~4]{Eremenko_zero_of_derievitive}
and \cite[Thm.~1.5]{Eremenko_Langley_Rossi_Sum_Cauchy_Kernels}.
\end{proof}

The following theorem is the only point in the paper where the classical
Wiman--Valiron theory is used. It describes the growth of transcendental
entire solutions of second-order differential equations with rational
coefficients in terms of the leading asymptotics of the coefficients.

\begin{theorem}\label{thm:wv-growth-ode}
Let $g$ be a transcendental entire solution of
\begin{equation*}\label{eq:equation_entire_classif}
g''(z)+A(z)g'(z)+B(z)g(z)=0,
\end{equation*}
where $A$ and $B$ are rational functions such that
\[
A(z)\sim a z^{\alpha},
\qquad
B(z)\sim b z^{\beta},
\qquad z\to\infty,
\]
for some $\alpha,\beta\in\mathbb Z$ and $a,b\in\mathbb C\setminus\{0\}$.
Then there exists a set $E\subset(0,\infty)$ of finite logarithmic measure
such that the following asymptotics hold as $r\to\infty$, $r\notin E$.

\smallskip
\noindent
\textup{(i)} If $\beta<2\alpha$, then $\alpha>-1$ and either
\[
\log M(r)\sim \frac{|a|}{\alpha+1}\,r^{\alpha+1},
\]
or
\[
\log M(r)=o\!\left(r^{\alpha+1}\right).
\]

\smallskip
\noindent
\textup{(ii)} If $\beta>2\alpha$, then $\beta>-2$ and
\[
\log M(r)\sim \frac{2\sqrt{|b|}}{\beta+2}\,r^{(\beta+2)/2}.
\]

\smallskip
\noindent
\textup{(iii)} If $\beta=2\alpha$, then $\alpha>-1$ and, for some choice of
the sign $\pm$,
\[
\log M(r)\sim
\frac{\bigl| -a\pm\sqrt{a^2-4b}\bigr|}{2(\alpha+1)}\,r^{\alpha+1}.
\]
\end{theorem}

\begin{proof}
Let $\nu(r)=\nu(r,g)$ denote the central index of $g$. By the classical
Wiman--Valiron theorem, see \cite[Ch.~3, Thm.~3.2]{LaineComplexDiffeq}, there exists a set $E\subset(0,\infty)$ of finite
logarithmic measure such that for every $r\notin E$, if $z_r$ satisfies
\[
|z_r|=r,
\qquad
|g(z_r)|=M(r),
\]
then
\[
\frac{g^{(l)}(z_r)}{g(z_r)}
=
\left(\frac{\nu(r)}{z_r}\right)^l(1+o(1)),
\qquad l=1,2.
\]
Substituting these asymptotics into
\eqref{eq:equation_entire_classif} and using
\[
A(z_r)=a z_r^{\alpha}(1+o(1)),
\qquad
B(z_r)=b z_r^{\beta}(1+o(1)),
\]
we obtain
\[
\nu(r)^2(1+o(1))
+
a z_r^{\alpha+1}\nu(r)(1+o(1))
+
b z_r^{\beta+2}(1+o(1))
=0.
\]
Hence
\[
\nu(r)
=
-\frac{a}{2}z_r^{\alpha+1}
\pm
\frac12\sqrt{a^2 z_r^{2\alpha+2}-4b z_r^{\beta+2}}\,(1+o(1)).
\]

\begin{enumerate}
\item Assume first that $\beta<2\alpha$. Then
\[
\sqrt{a^2 z_r^{2\alpha+2}-4b z_r^{\beta+2}}
=
a z_r^{\alpha+1}(1+o(1))
\]
for a branch of the square root. It follows that either
\[
\nu(r)=o(r^{\alpha+1})
\]
or
\[
\nu(r)=|a|\,r^{\alpha+1}(1+o(1)).
\]
Since $g$ is transcendental entire, $\nu(r)\to\infty$, and hence
$\alpha>-1$.

\item Next assume that $\beta>2\alpha$. Then
\[
\sqrt{a^2 z_r^{2\alpha+2}-4b z_r^{\beta+2}}
=
2\sqrt{-b}\,z_r^{(\beta+2)/2}(1+o(1))
\]
for a branch of square root. Consequently,
\[
\nu(r)=\sqrt{|b|}\,r^{(\beta+2)/2}(1+o(1)).
\]
Again $\nu(r)\to\infty$ implies $\beta>-2$.

\item Finally, let $\beta=2\alpha$. Then $a^2 z_r^{2\alpha+2}-4b z_r^{\beta+2}
=
(a^2-4b)z_r^{2\alpha+2}$,
and hence, for some choice of sign $\pm$,
\[
\nu(r)
=
\frac{-a\pm\sqrt{a^2-4b}}{2}\,
z_r^{\alpha+1}(1+o(1)).
\]
Therefore,
\[
\nu(r)\sim
\frac{\bigl|-a\pm\sqrt{a^2-4b}\bigr|}{2}\,r^{\alpha+1}.
\]
Since $\nu(r)\to\infty$, we again have $\alpha>-1$.
\end{enumerate}

To pass from $\nu(r)$ to $M(r)$, we use the standard relation between
the central index and the maximum modulus. It follows from the classical
argument in \cite[Ch.~6, \S9, Probl.~68]{PolyaSzegoII}, together with the
continuation lemma for inequalities \cite[Ch.~6, Lem.~2.1]{Book_Gold_Ostr}.
Namely, for $\sigma>0$ and $C>0$,
\[
\nu(r)\sim C r^\sigma
\quad\Longleftrightarrow\quad
\log M(r)\sim \frac{C}{\sigma}\,r^\sigma
\]
outside a set of finite logarithmic measure. Similarly,
\[
\nu(r)=o(r^\sigma)
\quad\Longleftrightarrow\quad
\log M(r)=o(r^\sigma).
\]
Applying this with $\sigma=\alpha+1$ in cases \textup{(i)} and
\textup{(iii)}, and with $\sigma=(\beta+2)/2$ in case \textup{(ii)}, we
obtain the asserted asymptotics for $\log M(r)$.
\end{proof}

As an immediate consequence, one obtains the growth parameters of entire
solutions of the equation naturally associated with the representation
$f=P/g^2$.

\begin{corollary}\label{cor:ode-growth-parameters}
Let $g$ be an entire solution of order at least $1$ of
\[
g''(z)+\frac{P'(z)}{P(z)}\,g'(z)+\frac{Q(z)}{P(z)}\,g(z)=0,
\]
where
\[
P(z)=z^p+\dots,\qquad
Q(z)=Cz^q+\dots,\qquad C\neq0
\]
are polynomials. Then
\[
q\ge p-1,
\]
and
\[
\rho(g)=\lambda(g)=\frac{q-p+2}{2},
\qquad
\sigma(g)=\tau(g)=\frac{2\sqrt{|C|}}{q-p+2}.
\]
\end{corollary}

\begin{proof}
Since $g$ has order at least $1$, it is transcendental. In the notation of
Theorem~\ref{thm:wv-growth-ode},
\[
A(z)=\frac{P'(z)}{P(z)}\sim \frac{p}{z},
\qquad
B(z)=\frac{Q(z)}{P(z)}\sim C z^{q-p},
\qquad z\to\infty.
\]
Thus,
\[
\alpha=-1,
\qquad
\beta=q-p.
\]
The alternatives \textup{(i)} and \textup{(iii)} in
Theorem~\ref{thm:wv-growth-ode} are impossible, since they require
$\alpha>-1$. Hence only alternative \textup{(ii)} can occur. It follows that $q-p>-2$,
and therefore, since $p$ and $q$ are integers, $q\ge p-1$.
The stated values of the order, lower order, type, and lower type follow
directly from the asymptotic formula in alternative \textup{(ii)}.
\end{proof}

\begin{proof}[Proof of Theorem~\ref{thm:finite-zero-reduction}]
Since $f$ has only finitely many zeros and only double poles, it can be
written in the form
\[
f=\frac{P}{g^2},
\]
where $P$ is a polynomial and $g$ is an entire function whose zeros are
simple and coincide with the poles $t_n$. Comparing the coefficient of
$(z-t_n)^{-1}$ in the Laurent expansion of $P/g^2$ at $t_n$, we obtain
\[
P'(t_n)g'(t_n)-P(t_n)g''(t_n)=0,
\qquad n\in\mathbb N.
\]
Thus
\[
Q:=\frac{P'g'-Pg''}{g}
\]
is entire, and hence
\[
Pg''-P'g'+Qg=0.
\]

Assume now that $f$ has finite order. Since $f=P/g^2$, we have
\[
\rho(g)=\rho(f),
\qquad
\lambda(g)=\lambda(f).
\]
Moreover,
\[
T(r,f)=2T(r,g)+O(\log r).
\]
For an entire function $g$, the standard inequalities
\[
T(r,g)\le \log^+ M(r,g)
\le
\frac{R+r}{R-r}\,T(R,g),
\qquad 0<r<R,
\]
relate the Nevanlinna characteristic and the maximum modulus; see
\cite[Ch.~I, \S7, Thm.~7.1]{Book_Gold_Ostr}. Thus the finite-order growth
information obtained below for $g$ transfers to $f$, in particular the
equality of type and lower type. By Proposition~\ref{prop:squares-order},
\[
\rho(g)=\rho(f)\ge1.
\]
Moreover,
\[
Q=P'\frac{g'}{g}-P\frac{g''}{g}.
\]
Hence the lemma on the logarithmic derivative gives
\[
m(r,Q)=O(\log r),
\qquad r\to\infty.
\]
Since $Q$ is entire, it follows
that $Q$ is a polynomial; see, for example,
\cite[Ch.~I]{Book_Gold_Ostr}.

We can now apply Corollary~\ref{cor:ode-growth-parameters} to $g$. It gives
\[
\deg Q\ge \deg P-1,
\qquad
\rho(g)=\lambda(g)=\frac{\deg Q-\deg P+2}{2},
\qquad
\sigma(g)=\tau(g).
\]
The equality $\deg Q=\deg P-1$ would imply $\rho(g)=1/2$, contrary to
$\rho(g)\ge1$. Thus
\[
\deg Q\ge\deg P,
\]
and consequently
\[
\rho(f)=\lambda(f)=\rho(g)=\lambda(g)
\in\left\{1,\frac32,2,\frac52,\dots\right\}.
\]
Moreover, since $T(r,f)$ and $\log M(r,g)$ are comparable up to harmless
polynomial factors, the equality $\sigma(g)=\tau(g)$ implies
\[
\sigma(f)=\tau(f)
\]
with the type understood in the Nevanlinna sense.

It remains to find the form of the rational coefficient $Q/P$. Let
$p_1,\dots,p_n$ be the distinct zeros of $P$, and let $\alpha_k$ be the
multiplicity of $p_k$. Since $g$ and $P$ have no common zeros, $Q/P$
has at most simple poles at the points $p_k$, and a direct calculation gives
\[
\operatorname*{Res}_{z=p_k}\frac{Q(z)}{P(z)}
=
\alpha_k\,\frac{g'(p_k)}{g(p_k)}.
\]
Thus
\[
\frac{Q(z)}{P(z)}
=
S(z)+\sum_{k=1}^{n}
\alpha_k\,\frac{g'(p_k)}{g(p_k)}\,\frac{1}{z-p_k},
\]
where $S$ is a polynomial.
\end{proof}

\subsection{Reduction to normal form}

We next reduce the equation for $g$ to a normal form suitable for
asymptotic analysis at infinity. The first step is the standard passage to normal
form; see, for example, \cite[Ch.~7]{Hille_ODE_complex}.

\begin{lemma}\label{lem:normal_form}
Let $\Omega\subset\mathbb C$ be simply connected, let $A$ and $B$ be
analytic in $\Omega$, and let $f$ be a solution of
\[
f''+Af'+Bf=0
\]
in $\Omega$. Fix $z_0\in\Omega$, and put
\[
u(z)=f(z)\exp\!\left(\frac12\int\limits_{z_0}^{z}A(w)\,dw\right).
\]
Then
\[
u''+Ru=0,
\qquad
R=-\frac12A'-\frac14A^2+B.
\]
\end{lemma}

\begin{proof}
A direct calculation gives the claim.
\end{proof}

For the equation arising from $f=P/g^2$, the normal form may be written
explicitly in terms of $P$ and $Q$.

\begin{corollary}\label{cor:normal_form_poly}
Let $P,Q$ be polynomials, and let $g$ be an entire solution of
\[
g''-\frac{P'}{P}\,g'+\frac{Q}{P}\,g=0
\]
in a simply connected domain $\Omega\subset\mathbb C$ containing no zeros
of $P$. Then, for any branch of
\[
u=\frac{g}{\sqrt P}
\]
in $\Omega$, one has
\[
u''+Ru=0,
\]
where
\[
R=\frac12\,\mathcal S(F)+\frac{Q}{P},
\qquad F'(z)=P(z),
\]
and $\mathcal S(F)$ is the Schwarzian derivative of $F$. Equivalently,
\[
R=
\frac12\!\left(\frac{P''}{P}-\frac32\Bigl(\frac{P'}{P}\Bigr)^2\right)
+\frac{Q}{P}.
\]
\end{corollary}

We shall also need the partial-fraction expansion of the Schwarzian term.

\begin{lemma}\label{lem:schwarzian-partial-fractions}
Let
\[
P(z)=\prod_{k=1}^{n}(z-p_k)^{m_k}
\]
be a polynomial. Then
\[
\mathcal S(F)
=
\frac{P''}{P}-\frac32\Bigl(\frac{P'}{P}\Bigr)^2
=
-\sum_{k=1}^{n}\frac{m_k(m_k+2)}{2\,(z-p_k)^2}
-
\sum_{k=1}^{n}\frac{M_k}{z-p_k},
\]
where
\[
M_k=
\frac12\sum_{\substack{j=1\\ j\neq k}}^{n}
\frac{m_km_j}{p_k-p_j},
\qquad
\sum_{k=1}^{n}M_k=0.
\]
\end{lemma}

\begin{proof}
Since
\[
\frac{P'}{P}=\sum_{k=1}^{n}\frac{m_k}{z-p_k},
\qquad
\left(\frac{P'}{P}\right)'=
-\sum_{k=1}^{n}\frac{m_k}{(z-p_k)^2},
\]
substituting into
\[
\mathcal S(F)=\left(\frac{P'}{P}\right)'-\frac12\left(\frac{P'}{P}\right)^2
\]
gives the stated expansion. The identity $\sum_{k=1}^{n}M_k=0$ follows by
antisymmetry under the interchange of indices.
\end{proof}

To prepare for the Liouville transformation, we now normalize the leading
part of the potential $R$ at infinity.

\begin{lemma}[Normalization under an affine change of variables]\label{lem:affine-normalization}
Assume that
\[
P(z)=z^p+a z^{p-1}+O(z^{p-2}),
\qquad
Q(z)=b z^q+c z^{q-1}+O(z^{q-2}),
\qquad b\neq0,
\]
and set $N=q-p\ge0$. Let $R$ be the coefficient in
Corollary~\ref{cor:normal_form_poly}. Then:

\smallskip
\noindent
\textup{(i)} if $N\ge1$, there exist $\mu>0$ and $\beta\in\mathbb C$
such that
\[
\mu^2 R(\mu z+\beta)=z^N\bigl(1+O(z^{-2})\bigr),
\qquad z\to\infty;
\]

\smallskip
\noindent
\textup{(ii)} if $N=0$, there exists $\mu>0$ such that
\[
\mu^2R(\mu z)=1+\frac{D}{z}+O(z^{-2}),
\qquad z\to\infty,
\]
where
\[
D=\frac1{\sqrt b}\sum_{k=1}^{p}D_k
\]
and the coefficients $D_k$ are those occurring in the partial-fraction
expansion of $Q/P$.
\end{lemma}

\begin{proof}
From the expansions of $P$ and $Q$ one obtains
\[
\frac{Q(z)}{P(z)}
=
\begin{cases}
z^N\!\left(b+\dfrac{c-ab}{z}+O(z^{-2})\right), & N\ge1,\\[1.2ex]
b+\dfrac{\sum_{k=1}^{p}D_k}{z}+O(z^{-2}), & N=0.
\end{cases}
\]
Since the Schwarzian term in Corollary~\ref{cor:normal_form_poly} is
$O(z^{-2})$ at infinity, the same asymptotics holds for $R$.

If $N\ge1$, then
\[
\mu^2R(\mu z+\beta)
=
b\mu^{N+2}z^N
+
\mu^{N+1}(bN\beta+c-ab)z^{N-1}
+
O(z^{N-2}).
\]
Choosing
\[
\mu=b^{-1/(N+2)},
\qquad
\beta=\frac{ab-c}{bN},
\]
gives
\[
\mu^2R(\mu z+\beta)=z^N\bigl(1+O(z^{-2})\bigr).
\]

If $N=0$, then
\[
\mu^2R(\mu z)=\mu^2b+\frac{\mu\sum_{k=1}^{p}D_k}{z}+O(z^{-2}),
\]
and the choice $\mu=b^{-1/2}$ yields the claim.
\end{proof}

Thus, after an affine change of variables, the normal form becomes a
perturbation of the model equation $u''+z^N u=0$.

\begin{corollary}[Reduction to normal form]\label{cor:normalized-equation}
Under the assumptions of Theorem~\ref{thm:finite-zero-reduction}, successive
applications of Corollary~\ref{cor:normal_form_poly} and
Lemma~\ref{lem:affine-normalization} reduce the equation
\[
g''-\frac{P'}{P}\,g'+\frac{Q}{P}\,g=0
\]
to
\begin{equation}\label{eq:ode}
u''+Ru=0,
\end{equation}
where
\[
u(z)=\frac{g(\mu z+\beta)}{\sqrt{P(\mu z+\beta)}}
\]
and
\[
R(z)=
\begin{cases}
z^N\bigl(1+O(z^{-2})\bigr), & N\ge1,\\[1ex]
1+\dfrac{D}{z}+O(z^{-2}), & N=0.
\end{cases}
\]
\end{corollary}

\subsection{The Liouville transformation}

After Corollary~\ref{cor:normalized-equation}, the equation has the form \eqref{eq:ode}, where $R$ is analytic near infinity and has the asymptotic form stated
there. We shall use the Liouville transformation
\[
w=\int\limits_{z_0}^{z}\sqrt{R(\zeta)}\,d\zeta,
\qquad
U(w)=R(z)^{1/4}u(z),
\]
which reduces \eqref{eq:ode} to a perturbation of the constant-coefficient
equation $U''+U=0$; see Hille \cite[Ch.~7]{Hille_ODE_complex}. The
presentation below follows closely Gundersen--Heittokangas--Zemirni
\cite{Gunder_Heittokangas_Asymp_Int_Theory}. In the new variable one obtains
\[
U''(w)+\bigl(1-T(w)\bigr)U(w)=0,
\]
where
\[
T(w)=\frac14\left(
\frac{R''(z)}{R(z)^2}
-\frac{5R'(z)^2}{4R(z)^3}
\right).
\]

For the rest of this subsection, let $\mathcal R$ be the class of functions
$R$ analytic near infinity for which, with some $N\in\mathbb N_0$,
\[
R(z)=
\begin{cases}
z^N\bigl(1+O(z^{-2})\bigr), & N\ge1,\\[1ex]
1+\dfrac{D}{z}+O(z^{-2}), & N=0,
\end{cases}
\qquad z\to\infty,
\]
where $D\in\mathbb C$. We write
\[
\ell(z)=\frac{2}{N+2}\,z^{\frac{N+2}{2}}.
\]
For $k=0,\dots,N+1$, $r>0$, and $\delta\ge0$, set
\[
\theta_k=\frac{2\pi k}{N+2},
\qquad
G_k(r,\delta)=
\bigl\{
z\in\mathbb C:\ |z|>r,\ 
\arg z\in(\theta_{k-1}+\delta,\theta_{k+1}-\delta)
\bigr\},
\]
and write $G_k(r)=G_k(r,0)$. Finally, put
\[
\varepsilon_{N,D}(r)=
\begin{cases}
r^{-2}, & N\ge3,\\[0.6ex]
(\log r)\,r^{-2}, & N=2,\\[0.6ex]
r^{-3/2}, & N=1,\\[0.6ex]
\bigl(1+|D|\log r\bigr)r^{-1}, & N=0.
\end{cases}
\]

For $r>0$ and $\delta>0$, put
\[
\Omega_k(r,\delta)=
\Bigl\{
w\in\mathbb C:\ |w|>r^{\frac{N+2}{2}},\
|\arg w-\pi k|<\pi-\frac{N+2}{2}\delta
\Bigr\}.
\]
The first estimate fixes the geometry of the Liouville map. It gives a
well-defined branch of the transformation in each sector and compares it with
the model map $\ell$. It is based on
\cite[Lem.~2.1,~2.2]{Gunder_Heittokangas_Asymp_Int_Theory}.

\begin{lemma}\label{lem:liouville-asymptotics}
Let $R\in\mathcal R$, and let $N=N(R)$. Then, for all sufficiently large
$r>0$ and each $k=0,\dots,N+1$, the sector $G_k(r)$ admits analytic
branches of $R^\alpha$, $\alpha\in\mathbb C$. In particular, after
choosing a branch of $\sqrt R$, the Liouville map
\begin{equation}\label{eq:liouville-map}
L_k(z)=\int\limits_{z_0}^{z}\sqrt{R(\zeta)}\,d\zeta,
\qquad z_0\in G_k(r),
\end{equation}
is well defined and analytic in $G_k(r)$.

Moreover, $r$ may be chosen so that $L_k$ is univalent in $G_k(r)$, and
\[
L_k(z)=
\frac{2}{N+2}z^{\frac{N+2}{2}}
\bigl(1+O(\varepsilon_{N,D}(|z|))\bigr),
\qquad z\to\infty,\quad z\in G_k(r).
\]

Finally, for every sufficiently small $\delta>0$ there exists
$r_0=r_0(\delta)>0$ such that, for all $r>r_0$ and each
$k=0,\dots,N+1$,
\[
G_k(c_1r,3\delta)
\subset
L_k^{-1}(\Omega_k(r,\delta))
\subset
G_k(c_2r,\delta),
\]
where $c_1,c_2>0$ depend only on $N$ and $\delta$.
\end{lemma}

Although \cite{Gunder_Heittokangas_Asymp_Int_Theory} states the quoted
estimates for polynomial $R$, the proofs use only the sectorial asymptotic
behaviour of $R$ at infinity, and hence apply to the class $\mathcal R$
considered here. The right-hand inclusion in the final assertion is not stated
explicitly there; it follows by the same argument from the sectorial
asymptotics of $L_k$.

We now apply the transformation to the differential equation. In the
$w$-plane, the equation becomes a small perturbation of $U''+U=0$. The
following form is obtained from
\cite[Lem.~2.3]{Gunder_Heittokangas_Asymp_Int_Theory}.

\begin{lemma}\label{lem:liouville-transform}
Let $R\in\mathcal R$, and let $\delta>0$ be sufficiently small. Then
there exist $c=c(N,\delta)>0$ and $r_0=r_0(\delta)>0$ such that, for
$r>r_0$ and $k=0,\dots,N+1$, the following holds. If $u$ is an analytic
solution of
\[
u''+Ru=0
\]
in $G_k(r)$, and
\[
w=L_k(z),
\qquad
U(w)=u(z)R(z)^{1/4},
\]
then $U$ is analytic in $\Omega_k(cr,\delta)$ and satisfies
\[
U''(w)+(1-T(w))U(w)=0,
\]
where
\[
T(w)=
\frac14\left(
\frac{R''(z)}{R(z)^2}
-\frac{5R'(z)^2}{4R(z)^3}
\right)
=O(w^{-2}),
\qquad
w\to\infty,\quad w\in\Omega_k(cr,\delta).
\]
\end{lemma}

Again, the cited result is stated for polynomial $R$, but the derivation of
the transformed equation and the estimate $T(w)=O(w^{-2})$ use only the
sectorial asymptotics at infinity. Thus the same argument applies to
$R\in\mathcal R$.

It remains to translate localization in the $w$-plane back to angular
localization in the original $z$-plane. For $w_0\in\mathbb C$ and
$\varepsilon>0$, put
\[
\Pi_k(w_0,\varepsilon)=
\{\,w\in\mathbb C:\ |\Im w-\Im w_0|\le\varepsilon,\ 
(-1)^k\Re (w-w_0)\ge 0\,\}.
\]
We also denote by
\[
\Lambda_k(r,C)=
\bigl\{
z\in\mathbb C:\ |z|>r,\ 
|\arg z-\theta_k|<C\,\varepsilon_{N,D}(|z|)
\bigr\}
\]
the angular neighbourhood of the Stokes ray
\[
\Gamma_k=\{\,z\in\mathbb C:\ \arg z=\theta_k\,\}.
\]

\begin{lemma}\label{lem:preimage-horizontal-ray}
Let $R\in\mathcal R$, and let $\delta>0$ be sufficiently small. Then, for
each $k=0,\dots,N+1$, there exists $\rho_0=\rho_0(\delta)>0$ such that,
for all $\rho>\rho_0$ and all $w_0\in\Omega_k(\rho,\delta)$,
\[
L_k^{-1}\!\bigl(\Pi_k(w_0,\varepsilon)\bigr)
\subset
\Lambda_k(c\rho,C),
\]
where $C>0$ depends on $|w_0|$, $\delta$, $\varepsilon$, and $N$, but
not on $\rho$.
\end{lemma}

The corresponding result in
\cite[Lem.~2.4]{Gunder_Heittokangas_Asymp_Int_Theory} is formulated for the
preimage of a horizontal ray. The half-strip version used here follows by the
same argument.

\subsection{Asymptotic integration in Stokes regions}

After the Liouville transformation, the problem reduces to a perturbation of
the equation $U''+U=0$. We record the corresponding asymptotic integration
in a form convenient for later use along horizontal rays.

For $r,\delta>0$, set
\[
D^{\pm}(r,\delta)
=
\bigl\{
w\in\mathbb C:\ |w|>r,\ |\arg(\pm w)|<\pi-\delta
\bigr\}.
\]
Let $0<\delta_0<\delta<\pi$, $r_0>0$, and let $F$ be analytic in
$D^{+}(r_0,\delta_0)$. We say that $F$ satisfies condition
$\mathcal F^{+}$ if, for every $w_0\in D^{+}(r_0,\delta_0)$, the integral
\[
\int\limits_0^{+\infty}|F(w_0+t)|\,dt
\]
converges and
\[
\sup_{w_0\in D^{+}(r,\delta)}
\int\limits_0^{+\infty}|F(w_0+t)|\,dt
\longrightarrow0,
\qquad r\to\infty.
\]
The condition $\mathcal F^{-}$ in $D^{-}(r_0,\delta_0)$ is defined
analogously, with integration along the negative real direction. In
particular, $\mathcal F^\pm$ holds whenever
\[
F(w)=O(w^{-2}),\qquad w\to\infty,
\]
in the corresponding domain $D^\pm(r_0,\delta_0)$.

The following results are standard consequences of the Volterra form of
Hille's asymptotic integration theory; see Hille \cite[Ch.~11]{Hille_ODE_complex}
and Gundersen--Heittokangas--Zemirni
\cite{Gunder_Heittokangas_Asymp_Int_Theory}.

\begin{theorem}[{\cite[Thm.~3.1]{Gunder_Heittokangas_Asymp_Int_Theory}}]\label{thm:volterra-solutions}
Let $F$ satisfy $\mathcal F^{+}$ in $D^{+}(r_0,\delta_0)$. Then
there exists $r_1>0$ such that in $D^{+}(r_1,\delta)$ the equation
\begin{equation}\label{eq:volterra-ode}
f''+(1-F)f=0
\end{equation}
has unique analytic solutions $E^\pm$ satisfying
\[
E^\pm(w)=
e^{\pm iw}
+\int\limits_0^{+\infty}\sin t\,F(w+t)\,E^\pm(w+ t)\,dt.
\]
Moreover, every solution of \eqref{eq:volterra-ode} in
$D^{+}(r_1,\delta)$ is of the form
\[
f=\lambda^+E^+ + \lambda^-E^-,
\qquad
\lambda^\pm\in\mathbb C.
\]
\end{theorem}

Here and below the Volterra results are stated in $D^+$; the statements in
$D^-$ are identical after the corresponding change of signs. In either
domain, $E^+$ and $E^-$ denote the two canonical solutions asymptotic to
$e^{iw}$ and $e^{-iw}$, respectively.

We shall also use the corresponding asymptotic form of these solutions. The
asymptotic formula for $E^\pm$ and the assertion on zeros are
\cite[Cor.~3.1]{Gunder_Heittokangas_Asymp_Int_Theory}; the estimate for the
derivative follows from the same Volterra integral equations and Gronwall's
lemma.

\begin{corollary}\label{cor:volterra-asymptotics}
As $w\to\infty$ in $D^{+}(r_1,\delta)$,
\begin{align*}
E^\pm(w)
&=
e^{\pm iw}
\left(
1+O\!\left(\int\limits_{0}^{+\infty}|F(w+t)|\,dt\right)
\right),\\[1ex]
(E^\pm)'(w)
&=
\pm i\,e^{\pm iw}
\left(
1+O\!\left(\int\limits_{0}^{+\infty}|F(w+t)|\,dt\right)
\right).
\end{align*}
In particular, $E^\pm$ has only finitely many zeros in
$D^{+}(r_1,\delta)$ for sufficiently large $r_1$.
\end{corollary}

For $\varepsilon>0$ and $w_0\in\mathbb C$, let
\[
Q_\varepsilon(w_0)=
\{\,w\in\mathbb C:\ |\Re w-\Re w_0|<\varepsilon,\
|\Im w-\Im w_0|<\varepsilon\,\}.
\]

The following statement is \cite[Lem.~3.2]{Gunder_Heittokangas_Asymp_Int_Theory},
with the derivative asymptotic added.

\begin{theorem}\label{thm:volterra-zero-pattern}
Let $\varepsilon>0$ be sufficiently small, and let
\[
f=\lambda^+E^+ + \lambda^-E^-
\]
be a nontrivial solution of \eqref{eq:volterra-ode} in
$D^{+}(r_1,\delta)$.

If $\lambda^+\lambda^-\neq0$, then there exist
$w_0\in D^{+}(r_1,\delta)$, $n_0\in\mathbb N$, and $\eta\in\mathbb C$
such that all but finitely many zeros of $f$ lie in the squares
\[
Q_\varepsilon(w_0)+\pi n,
\qquad n\ge n_0,
\]
each of these squares contains exactly one zero $\eta_n$, and
\[
\eta_n=\pi n+\eta+o(1),
\qquad n\to\infty.
\]
Moreover,
\[
\lim_{n\to\infty}(-1)^n f'(\eta_n)=2i\,\lambda^+e^{i\eta}.
\]

If $\lambda^+\lambda^-=0$, then $f$ has only finitely many zeros in
$D^{+}(r_1,\delta)$.
\end{theorem}

\begin{proof}
The assertions on the location and number of zeros are precisely
\cite[Lem.~3.2]{Gunder_Heittokangas_Asymp_Int_Theory}. It remains only to
derive the displayed asymptotic for $f'(\eta_n)$.

At a zero $\eta_n$ we have
\[
\lambda^-E^-(\eta_n)=-\lambda^+E^+(\eta_n).
\]
By Corollary~\ref{cor:volterra-asymptotics},
\[
(E^\pm)'(\eta_n)=\pm i E^\pm(\eta_n)+o(1),
\]
and hence
\begin{align*}
f'(\eta_n)
&=\lambda^+(E^+)'(\eta_n)+\lambda^-(E^-)'(\eta_n)\\
&=i\lambda^+E^+(\eta_n)-i\lambda^-E^-(\eta_n)+o(1)\\
&=2i\,\lambda^+E^+(\eta_n)+o(1).
\end{align*}
Since
\[
E^+(\eta_n)=e^{i\eta_n}(1+o(1))
\]
and
\[
\eta_n=\pi n+\eta+o(1),
\]
we obtain
\[
E^+(\eta_n)=(-1)^n e^{i\eta}(1+o(1)).
\]
Therefore
\[
f'(\eta_n)=(-1)^n2i\,\lambda^+e^{i\eta}(1+o(1)),
\]
which proves the asserted limit.
\end{proof}

\subsection{Zeros of entire solutions}

We now return to the original variable $z$. The results of the previous
subsections give the asymptotic behaviour of solutions of
\begin{equation}\label{eq:osnovnoe}
g''-\frac{P'}{P}g'+\frac{Q}{P}g=0,
\end{equation}
where $P$ and $Q$ are polynomials. We write
\[
P(z)=z^p+a z^{p-1}+O(z^{p-2}),
\qquad
Q(z)=b z^q+c z^{q-1}+O(z^{q-2}),
\qquad b\neq0,
\]
with $q\ge p\ge0$, and put
\[
N=q-p.
\]
The polynomial part of $Q/P$ is denoted by $S$; thus, if
$p_1,\dots,p_p$ are the zeros of $P$, counted with multiplicities, then
\[
\frac{Q(z)}{P(z)}
=
S(z)+\sum_{j=1}^{p}\frac{D_j}{z-p_j},
\]
where
\[
S(z)=b z^N+(c-ab)z^{N-1}+O(z^{N-2}).
\]
The normal form of \eqref{eq:osnovnoe} involves the potential
\[
R
=
\frac12\left(
\frac{P''}{P}
-\frac32\left(\frac{P'}{P}\right)^2
\right)
+\frac{Q}{P}.
\]

The Stokes geometry is determined by the leading term of $R$ at infinity.
For $N\ge1$, we use the shift
\[
d=\frac{ab-c}{bN},
\]
while for $N=0$ we put $d=0$. The Stokes directions are
\[
\theta_k=\frac{2\pi k}{N+2}-\frac{\arg b}{N+2},
\qquad k=0,\dots,N+1,
\]
if $N\ge1$, and
\[
\theta_k=\pi k-\frac{\arg b}{2},
\qquad k=0,1,
\]
if $N=0$. The corresponding Stokes rays, issuing from the point $d$, are
\[
\Gamma_k=\{\,z\in\mathbb C:\ \arg (z-d)=\theta_k\,\}.
\]
We write
\[
S_k=\{\,z\in\mathbb C:\ \arg (z-d)\in(\theta_k,\theta_{k+1})\,\}
\]
for the Stokes sectors. For $r>0$ and $\delta>0$, set
\[
G_k(r,\delta)=
\left\{
z\in\mathbb C:\ |z-d|>r,\ 
\arg (z-d)\in(\theta_{k-1}+\delta,\theta_{k+1}-\delta)
\right\}.
\]
Finally, for $r>0$ and $C>0$, put
\[
\Lambda_k(r,C)=
\left\{
z\in\mathbb C:\ |z-d|>r,\ 
|\arg (z-d)-\theta_k|
<C\,\varepsilon_{N,D}(|z-d|)
\right\}.
\]
Here, in the case $N=0$, $D$ denotes the coefficient of the $1/z$-term
in the corresponding normalized expansion
\[
R(z)=1+\frac{D}{z}+O(z^{-2})
\]
after the affine normalization used in the Liouville transformation.

The next theorem is obtained from the sectorial zero localization theorem of
Gundersen--Heittokangas--Zemirni
\cite[Thm.~1]{Gunder_Heittokangas_Asymp_Int_Theory}, after reducing
\eqref{eq:osnovnoe} to normal form. The extension needed here is to
rational-coefficient equations arising from the pair $(P,Q)$. It also
exhibits the borderline phenomenon $N=0$: in this case the zeros are
localized only in logarithmic neighbourhoods of the Stokes rays, and need not
approach them in the Euclidean metric.

\begin{theorem}\label{thm:zeros-stokes-localization}
Let $P$ and $Q$ be polynomials satisfying the assumptions and notation
introduced above, and let $g$ be a nontrivial entire solution of
\eqref{eq:osnovnoe}. Then, for each $k=0,\dots,N+1$, there exist
$r>0$, $\delta>0$, and a fundamental pair of analytic solutions
$g_k^+,g_k^-$ of \eqref{eq:osnovnoe} in $G_k(r,\delta)$ such that
\[
g=\lambda_k^+g_k^+ + \lambda_k^-g_k^-
\]
there, with some constants $\lambda_k^+,\lambda_k^-\in\mathbb C$. Moreover,
as $z\to\infty$, $z\in G_k(r,\delta)$,
\begin{align*}
g_k^\pm(z)
&=
(z-d)^{\frac{3p-q}{4}}
\exp\!\left(
\frac{\pm 2i\sqrt b}{N+2}\,
(z-d)^{\frac{N+2}{2}}(1+o(1))
\right)
\bigl(1+O((z-d)^{-1})\bigr),\\[1ex]
(g_k^\pm)'(z)
&=
\pm i\sqrt b\,(z-d)^{\frac{p+q}{4}}
\exp\!\left(
\frac{\pm 2i\sqrt b}{N+2}\,
(z-d)^{\frac{N+2}{2}}(1+o(1))
\right)
\bigl(1+O((z-d)^{-1})\bigr).
\end{align*}

The zero distribution is determined by the two coefficients:
$g$ has infinitely many zeros in $G_k(r,\delta)$ if and only if
\[
\lambda_k^+\lambda_k^-\neq0.
\]
In this case all but finitely many zeros $t_n$ of $g$ in $G_k(r,\delta)$
belong to $\Lambda_k(r,C)$ for some $C>0$, and
\[
t_n-d\sim
\left(\frac{\pi^2(N+2)^2}{4b}\right)^{\!\frac1{N+2}}
e^{\frac{2\pi i k}{N+2}}\,n^{\frac{2}{N+2}},
\qquad n\to\infty,
\]
while
\[
g'(t_n)\sim (-1)^n a_k\,n^{\frac{p+q}{2(N+2)}},
\qquad n\to\infty,
\]
where $a_k\neq0$ depends on $k$ and on the coefficients
$\lambda_k^\pm$. Furthermore,
\[
\mathrm{dist}(t_n,\Gamma_k)=
\begin{cases}
O(|t_n|^{-1}), & N\ge3,\\[0.8ex]
O\!\left(\dfrac{\log|t_n|}{|t_n|}\right), & N=2,\\[1.2ex]
O(|t_n|^{-1/2}), & N=1,\\[0.8ex]
O(|D|\log|t_n|+1), & N=0.
\end{cases}
\]
In particular, for $N>0$ the zeros approach the Stokes ray $\Gamma_k$ in
the Euclidean metric, whereas for $N=0$ one can guarantee in general only
logarithmic localization.
\end{theorem}

\begin{proof}
Fix $k\in\{0,\dots,N+1\}$, and choose $\delta>0$ sufficiently small. By
Corollary~\ref{cor:normalized-equation}, after the affine normalization and
the substitution
\[
u(z)=\frac{g(\mu z+\beta)}{\sqrt{P(\mu z+\beta)}},
\]
the equation \eqref{eq:osnovnoe} is reduced, in a sector of the form
$G_k(r,\delta)$ with $r$ sufficiently large, to
\begin{equation}\label{eq:norm_form}
u''+Ru=0,
\qquad R\in\mathcal R.
\end{equation}
Applying Lemma~\ref{lem:liouville-transform}, we pass to the Liouville
variables
\[
w=L_k(z),
\qquad
U_k(w)=u(z)R(z)^{1/4},
\]
and obtain
\[
U_k''+(1-T)U_k=0,
\qquad
T(w)=O(w^{-2}),
\]
in a domain $\Omega_k(r',\delta')$. Hence $T$ satisfies one of the
conditions $\mathcal F^\pm$, according to the parity of $k$. By
Theorem~\ref{thm:volterra-solutions},
\[
U_k=\lambda_k^+E_k^+ + \lambda_k^-E_k^- .
\]

The corresponding solutions of \eqref{eq:osnovnoe} are obtained by reversing
the Liouville transformation and the affine normalization:
\[
g_k^\pm(\mu z+\beta)
=
\sqrt{P(\mu z+\beta)}\,R(z)^{-1/4}E_k^\pm(L_k(z)).
\]
Thus
\[
g=\lambda_k^+g_k^+ + \lambda_k^-g_k^-
\]
in the corresponding sector. Corollary~\ref{cor:volterra-asymptotics} then
gives, as $z\to\infty$, $z\in G_k(r,\delta)$,
\begin{align*}
g_k^\pm(z)
&=
(z-d)^{\frac{3p-q}{4}}
\exp\!\left(
\frac{\pm 2i\sqrt b}{N+2}\,
(z-d)^{\frac{N+2}{2}}(1+o(1))
\right)
\bigl(1+O((z-d)^{-1})\bigr),\\[1ex]
(g_k^\pm)'(z)
&=
\pm i\sqrt b\,(z-d)^{\frac{p+q}{4}}
\exp\!\left(
\frac{\pm 2i\sqrt b}{N+2}\,
(z-d)^{\frac{N+2}{2}}(1+o(1))
\right)
\bigl(1+O((z-d)^{-1})\bigr).
\end{align*}

By Theorem~\ref{thm:volterra-zero-pattern}, $U_k$, and hence also $g$,
has only finitely many zeros in $G_k(r,\delta)$ if
$\lambda_k^+\lambda_k^-=0$. If $\lambda_k^+\lambda_k^-\neq0$, the same
theorem gives zeros $w_n$ of $U_k$ such that
\[
w_n=\pm\pi n+\eta_k+o(1),
\qquad
(U_k)'(w_n)\sim(-1)^n c_k,
\qquad c_k\neq0.
\]
Since $L_k(z)\sim\ell(z)$, where
\[
\ell(z)=\frac{2}{N+2}z^{\frac{N+2}{2}},
\]
Lemma~\ref{lem:liouville-asymptotics} gives, for the corresponding zeros
$z_n=L_k^{-1}(w_n)$ in the normalized variable,
\[
z_n\sim
\left(\frac{\pi^2(N+2)^2}{4}\right)^{\!\frac1{N+2}}
e^{\frac{2\pi i k}{N+2}}\,n^{\frac{2}{N+2}}.
\]
Undoing the affine normalization yields
\[
t_n-d\sim
\left(\frac{\pi^2(N+2)^2}{4b}\right)^{\!\frac1{N+2}}
e^{\frac{2\pi i k}{N+2}}\,n^{\frac{2}{N+2}}.
\]

It remains only to record the derivative at these zeros. From the inverse
Liouville transformation and the relation $u=g/\sqrt P$, differentiation at
a zero gives
\[
g'(t_n)\sim
C_k (t_n-d)^{\frac{3p-q}{4}}\,L_k'(t_n)\,(U_k)'(w_n),
\qquad C_k\neq0,
\]
with $C_k$ constant, since $U_k(w_n)=0$. Moreover,
\[
L_k'(t_n)\sim C_k'(t_n-d)^{N/2},
\qquad C_k'\neq0,\qquad N=q-p.
\]
Hence
\[
g'(t_n)
\sim
(-1)^n a_k (t_n-d)^{\frac{p+q}{4}},
\qquad a_k\neq0.
\]
Here $a_k$ is a constant depending on the sector and on the solution.
Together with the asymptotic formula for $t_n-d$, this gives
\[
g'(t_n)\sim (-1)^n a_k\,n^{\frac{p+q}{2(N+2)}},
\]
after changing the non-zero constant $a_k$.

Finally, the distance estimates follow from
Lemma~\ref{lem:preimage-horizontal-ray}: the preimage of a horizontal strip in
the $w$-plane is contained in an angular neighborhood of $\Gamma_k$ of
width $O(\varepsilon_{N,D}(|z-d|))$. This gives the stated bounds for
$N\ge3$, $N=2$, $N=1$, and $N=0$, respectively.
\end{proof}

The next corollary will be used repeatedly in the examples. We state it for
one Stokes sequence; since there are only finitely many Stokes sectors, the
same summability criterion is unchanged if all poles are taken together.

\begin{corollary}\label{cor:coefficient-summability}
In the notation of Theorem~\ref{thm:zeros-stokes-localization}, let
$\{t_n\}$ be a sequence of zeros of $g$ lying in one sector
$G_k(r,\delta)$ for which $\lambda_k^+\lambda_k^-\neq0$, and put
\[
c_n=\frac{P(t_n)}{g'(t_n)^2}.
\]
Then, for every $\alpha\in\mathbb R$,
\[
\frac{|c_n|}{|t_n|^{1+\alpha}}
\asymp
n^{-1-\frac{2\alpha}{N+2}},
\qquad n\to\infty.
\]
Consequently,
\[
\sum_{n=1}^{\infty}\frac{|c_n|}{|t_n|^{1+\alpha}}<\infty
\quad\Longleftrightarrow\quad
\alpha>0.
\]
The same equivalence holds for the union of all Stokes sequences.
\end{corollary}

\begin{proof}
By Theorem~\ref{thm:zeros-stokes-localization},
\[
|t_n|\asymp n^{\frac{2}{N+2}},
\qquad
|g'(t_n)|\asymp n^{\frac{p+q}{2(N+2)}}.
\]
Since $|P(t_n)|\asymp |t_n|^p$, we have
\[
|c_n|
=
\frac{|P(t_n)|}{|g'(t_n)|^2}
\asymp
n^{\frac{2p}{N+2}-\frac{p+q}{N+2}}
=
n^{-\frac{N}{N+2}}.
\]
Therefore
\[
\frac{|c_n|}{|t_n|^{1+\alpha}}
\asymp
n^{-\frac{N}{N+2}}\,
n^{-\frac{2(1+\alpha)}{N+2}}
=
n^{-1-\frac{2\alpha}{N+2}},
\]
which gives the criterion for one Stokes sequence. Since the total set of
poles is a finite union of such sequences, together with at most finitely many
exceptional poles, the convergence criterion is unchanged.
\end{proof}

In particular, the series $\sum_{n=1}^{\infty}\frac{c_n}{(z-t_n)^2}$
associated with such a solution is always locally uniformly convergent away
from the poles.

We shall say that $g$ has \emph{finitely many zeros near the Stokes ray}
$\Gamma_k$ if there exists $C>0$ such that $\Lambda_k(C)$ contains only
finitely many zeros of $g$; otherwise $g$ has infinitely many zeros
near $\Gamma_k$.

The next theorem relates the zeros in a Stokes sector to the sectorial limiting
behaviour of the corresponding solution. It follows from the results of
Hellerstein--Rossi \cite[Thm.~C]{Hellerstein_Gundersen_for_defect} and
Gundersen--Heittokangas--Zemirni
\cite[Thm.~2]{Gunder_Heittokangas_Asymp_Int_Theory}. The assertions involving
$g'$ and its limiting behaviour are obtained in the same way, using the
derivative asymptotics from the asymptotic integration results above.

\begin{theorem}\label{thm:sectorial-behavior}
Let $P$ and $Q$ be polynomials satisfying the assumptions and notation
introduced above, and let $g\not\equiv0$ be an entire solution of
\eqref{eq:osnovnoe}. Then, in each Stokes sector $S_k$ and for every
sufficiently small $\delta>0$, there exist a sign $\pm$ and a constant
$A_k\neq0$ such that, as $z\to\infty$, $z\in S_k(\delta)$,
\[
g(z)=A_k g_k^\pm(z)\bigl(1+O((z-d)^{-1})\bigr),
\qquad
g'(z)=A_k (g_k^\pm)'(z)\bigl(1+O((z-d)^{-1})\bigr).
\]
In particular, in each smaller Stokes sector either
\[
g(z)\to0,\qquad g'(z)\to0,
\]
or
\[
g(z)\to\infty,\qquad g'(z)\to\infty.
\]

For two adjacent Stokes sectors $S_k$ and $S_{k+1}$, the possible limiting
patterns for $g$ are precisely
\[
(0,\infty),\qquad (\infty,\infty),\qquad (\infty,0).
\]
Moreover, $g$ has infinitely many zeros near the common Stokes ray
$\Gamma_k$ if and only if the corresponding pattern is
$(\infty,\infty)$.

Let $k(g)$ be the number of Stokes rays near which $g$ has only finitely
many zeros. Equivalently, $k(g)$ is twice the number of Stokes sectors in
which $g$ tends to $0$ at infinity. Then $k(g)$ is even, and
\begin{align*}
n(r,0,g)
&\sim
\frac{2\sqrt{|b|}}{\pi(N+2)}
\bigl(N+2-2k(g)\bigr)r^{\frac{N+2}{2}},\\[1ex]
N(r,0,g)
&\sim
\frac{4\sqrt{|b|}}{\pi(N+2)^2}
\bigl(N+2-2k(g)\bigr)r^{\frac{N+2}{2}},\\[1ex]
T(r,g)
&\sim
\frac{4\sqrt{|b|}}{\pi(N+2)^2}
\left(N+2-\frac{k(g)}2\right)r^{\frac{N+2}{2}},
\end{align*}
as $r\to\infty$. Consequently,
\[
\delta(0,g)=\Delta(0,g)=\frac{k(g)}{2N+4-k(g)},
\qquad
\delta(w_0,g)=\Delta(w_0,g)=0
\quad (w_0\neq0).
\]
\end{theorem}

We may now characterize those entire solutions of \eqref{eq:osnovnoe} that
give rise to decompositions into sums of squares of Cauchy kernels.

\begin{theorem}\label{thm:razloz_v_sum_1}
Under the assumptions of Theorems~\ref{thm:zeros-stokes-localization} and
\ref{thm:sectorial-behavior}, let $g$ be an entire solution of \eqref{eq:osnovnoe}
having no common zeros with $P$. Then the following are equivalent:
\begin{enumerate}
\item[\textup{(i)}]
there exists a representation
\begin{equation}\label{eq:razloz}
\frac{P(z)}{g^2(z)}
=
\sum_{n=1}^{\infty}\frac{c_n}{(z-t_n)^2};
\end{equation}
\item[\textup{(ii)}]
$g(z)\to\infty$ in every Stokes sector as $z\to\infty$;
\item[\textup{(iii)}]
$g$ has infinitely many zeros near every Stokes ray.
\end{enumerate}
\end{theorem}

\begin{proof}
The equivalence of \textup{(ii)} and \textup{(iii)} follows directly from
Theorem~\ref{thm:sectorial-behavior}.

Assume \textup{(i)}. Since the zeros $t_n$ of $g$ are asymptotically
localized near the Stokes rays, the series on the right-hand side of
\eqref{eq:razloz} tends to $0$ in each Stokes sector. Hence
$g(z)\to\infty$ in every Stokes sector, proving
\textup{(i)}$\Rightarrow$\textup{(ii)}.

Assume now \textup{(ii)}, and set
\[
H(z)=\frac{P(z)}{g^2(z)}-\sum_{n=1}^{\infty}\frac{c_n}{(z-t_n)^2}.
\]
By Corollary~\ref{cor:coefficient-summability}, the series converges and $H$
is entire. Moreover, $H$ has order at most $\rho(g)=\frac{N+2}{2}$, by
Theorem~\ref{thm:zeros-stokes-localization} and Lemma~\ref{lem:equality_orders}. Since
$g(z)\to\infty$ in every Stokes sector, one has $H(z)\to0$ in every
Stokes sector of opening $\pi/\rho(g)$. The Phragmén--Lindelöf principle
therefore yields $H\equiv0$, and \eqref{eq:razloz} follows. Thus
\textup{(ii)}$\Rightarrow$\textup{(i)}.
\end{proof}

\section{The Schwarzian Reformulation and the Decomposition Criterion}

The asymptotic analysis of the previous section yields a sectorial criterion
for representability. We now reformulate the problem in terms of the
Schwarzian derivative. This isolates the global analytic obstruction in a
canonical form, namely the Laine condition, and prepares the algebraic
parametrization in the next section.

\subsection{The Schwarzian reformulation}

For a locally injective meromorphic function $F$, we write
\[
\mathcal S(F):=
\frac{F'''}{F'}-\frac32\left(\frac{F''}{F'}\right)^2
\]
for its Schwarzian derivative. Recall that $F$ is locally injective if and
only if $F'(z)\neq0$ at every finite point of its domain; in particular, all
zeros and poles of such a function are simple. The Schwarzian derivative
determines a locally injective meromorphic function up to postcomposition by a
fractional-linear map; see, for example,
\cite[Ch.~6]{LaineComplexDiffeq}.

\begin{lemma}\label{lem:schwarzian-mobius}
Let $G\subset\mathbb C$ be simply connected, and let $F_1,F_2$ be locally
injective meromorphic functions in $G$. Then
\[
\mathcal S(F_1)=\mathcal S(F_2)
\]
in $G$ if and only if there exists a fractional-linear transformation
\[
F_2=\frac{aF_1+b}{cF_1+d},
\qquad ad-bc\neq0.
\]
\end{lemma}

We shall use the following theorem of Laine on the global meromorphic
solvability of the Schwarzian equation; see
\cite[Ch.~6, Thm.~6.7]{LaineComplexDiffeq}.

\begin{theorem}[Laine]\label{thm:laine}
Let $G\subset\mathbb C$ be simply connected, and let $R$ be meromorphic in
$G$. Then the equation
\[
\mathcal S(F)=2R
\]
has a meromorphic solution $F$ in $G$ if and only if, for every pole
$z_0$ of $R$, the Laurent expansion of $R$ has the form
\[
R(z)=
\frac{1-m^2}{4\,(z-z_0)^2}
+\frac{b_1}{z-z_0}
+b_2+b_3(z-z_0)+\cdots,
\]
where $m=m(z_0)\in\mathbb N$, $m\ge2$, and the determinant
\[
\mathcal{D}_m(z_0)=0
\]
vanishes, where
\[
\mathcal D_m(z_0)
=
\det
\begin{pmatrix}
1-m & 0 & 0 & \cdots & 0 & b_1\\
b_1 & 4-2m & 0 & \cdots & 0 & b_2\\
b_2 & b_1 & 9-3m & \cdots & 0 & b_3\\
\vdots & \vdots & \vdots & \ddots & \vdots & \vdots\\
b_{m-2} & b_{m-3} & b_{m-4} & \cdots & (m-1)^2-(m-1)m & b_{m-1}\\
b_{m-1} & b_{m-2} & b_{m-3} & \cdots & b_1 & b_m
\end{pmatrix}.
\]
We shall say that $2R$ satisfies the \emph{Laine condition} in $G$. If
$G=\mathbb C$, the domain is omitted from the notation.
\end{theorem}

For a meromorphic function $f$, we also write
\[
\mathcal A(f):=\frac{f''}{f}-\frac32\left(\frac{f'}{f}\right)^2.
\]
If $F'=f$, then $\mathcal A(f)=\mathcal S(F)$. The next lemma records the
equivalent forms of the differential equation used below.

\begin{lemma}\label{lem:equivalent-formulations}
Let $P\not\equiv0$ and $Q$ be polynomials. On every simply connected
domain on which the relevant branches and quotients are defined, the following
formulations are equivalent:
\begin{equation}\label{eq:main-pq-equation}
g''-\frac{P'}{P}g'+\frac{Q}{P}g=0,
\end{equation}
\begin{gather*}
h''+Rh=0,
\qquad h=\frac{g}{\sqrt P},\\[0.6ex]
\mathcal A(f)=\mathcal A(P)+\frac{2Q}{P},
\qquad f=\frac{P}{g^2}=\frac1{h^2},\\[0.6ex]
F'-\frac12F^2=\mathcal A(P)+\frac{2Q}{P},
\qquad F=\frac{f'}{f}.
\end{gather*}
Here
\begin{equation}\label{eq:normal-potential-R}
R=\frac12\mathcal A(P)+\frac{Q}{P}.
\end{equation}
\end{lemma}

\begin{proof}
The equivalence of the first two equations is exactly
Corollary~\ref{cor:normal_form_poly}. Since $f=1/h^2$, a direct
calculation gives
\[
\mathcal A(f)=2R.
\]
Using the explicit expression for $R$ from
Corollary~\ref{cor:normal_form_poly}, we obtain
\[
\mathcal A(f)=\mathcal A(P)+\frac{2Q}{P}.
\]
Finally, for $F=f'/f$,
\[
F'-\frac12F^2
=
\frac{f''}{f}-\frac32\left(\frac{f'}{f}\right)^2
=
\mathcal A(f),
\]
which yields the last equation.
\end{proof}

The equation \eqref{eq:main-pq-equation}
has the special feature that, locally, reduction of order is governed by the
differential $P g^{-2}\,dz$. Indeed, if $g$ is a nonzero local solution
and $G$ is a local primitive of $P/g^2$, then $gG$ is a second local
solution. In general $G$ need not be single-valued. The point of the Laine
condition below is precisely to characterize when such primitives are
meromorphic in the plane.

The following local computation will be used to compare the poles of the
Schwarzian term with the zero or pole orders of $f$.

\begin{lemma}\label{lem:local-A-expansion}
Let $n\in\mathbb Z$, and let $f$ be analytic near $z_0$ with
$f(z_0)\neq0$. Then
\[
\mathcal A\!\bigl((z-z_0)^n f(z)\bigr)
=
-\frac{n(n+2)}{2\,(z-z_0)^2}
+O\!\left(\frac1{z-z_0}\right),
\qquad z\to z_0.
\]
\end{lemma}

\begin{proof}
Write
\[
u(z)=(z-z_0)^n f(z),
\qquad
\frac{u'(z)}{u(z)}=\frac{n}{z-z_0}+O(1),
\qquad
\frac{u''(z)}{u(z)}=\frac{n(n-1)}{(z-z_0)^2}
+O\!\left(\frac1{z-z_0}\right),
\]
and substitute these estimates into the definition of $\mathcal A(u)$.
\end{proof}

\subsection{The Laine condition and entire solutions}

We now isolate the global consequence of the Schwarzian reformulation. Local
primitives of $P/g^2$ always exist away from the zeros of $Pg$, but they
may have monodromy. The following theorem shows that, for the present class of
equations, the Laine condition is exactly the condition which removes this
global obstruction.

\begin{theorem}\label{thm:laine-entire-solutions}
Let $P$ and $Q$ be as above. Then the following are equivalent:
\begin{enumerate}
\item[\textup{(i)}]
there exists a local solution $g$ of \eqref{eq:main-pq-equation} and a local
primitive $G$ of $P/g^2$ which extends to a meromorphic function in
$\mathbb C$;

\item[\textup{(ii)}]
the meromorphic function $2R$, with $R$ defined by
\eqref{eq:normal-potential-R}, satisfies the Laine condition of
Theorem~\ref{thm:laine};

\item[\textup{(iii)}]
every local solution of \eqref{eq:main-pq-equation} extends to an entire function
$g$, and for every such $g$ the differential $P g^{-2}\,dz$ has a
meromorphic primitive in $\mathbb C$.
\end{enumerate}
\end{theorem}

\begin{proof}
Assume \textup{(i)}. Let $G$ be the meromorphic continuation of the local
primitive. By Lemma~\ref{lem:equivalent-formulations}, $\mathcal S(G)=2R$
on the initial domain, and hence in the plane. Thus, the Schwarzian equation
$\mathcal S(F)=2R$ has a meromorphic solution in $\mathbb C$, so $2R$
satisfies the Laine condition by Theorem~\ref{thm:laine}.

Assume \textup{(ii)}, and let $g$ be any local solution of
\eqref{eq:main-pq-equation} near a point where $Pg\neq0$. Choose a local primitive
$G$ of $P/g^2$. Then $\mathcal S(G)=2R$
by Lemma~\ref{lem:equivalent-formulations}. Since $2R$ satisfies the Laine
condition, there exists a meromorphic function $F$ in $\mathbb C$ such
that $\mathcal S(G)=2R$.
On the initial domain, $F$ and $G$ have the same Schwarzian derivative;
hence they differ by a fractional-linear transformation. Therefore $G$
extends meromorphically to $\mathbb C$, and so does
\[
f:=G'=\frac{P}{g^2}.
\]

It remains to read off the local behavior of $g$ from the meromorphic
function $f$. Let $z_*$ be a point which is not a zero of $P$. If $f$
has zero/pole order $m\in\mathbb Z$ at $z_*$, then
Lemma~\ref{lem:local-A-expansion}, applied to $\mathcal A(f)=2R$, gives $m=0$ or $m=-2$,
because $2R$ has no pole at $z_*$. Hence $P/f$ has order $0$ or $2$
at $z_*$, and $g^2=P/f$ has a holomorphic square root there.

Now let $z_*=p_k$, where $p_k$ is a zero of $P$ of multiplicity
$\alpha_k$. The principal part of $2R$ at $p_k$ is
\[
-\frac{\alpha_k(\alpha_k+2)}{2(z-p_k)^2}.
\]
Comparison with Lemma~\ref{lem:local-A-expansion} gives $m=\alpha_k$ or $m=-2-\alpha_k$.
Thus, $P/f$ has order $0$ or $2\alpha_k+2$ at $p_k$, and again
$g^2=P/f$ has a holomorphic square root. Hence $g$ extends holomorphically
across every point of the plane. Since the original local solution $g$ was
arbitrary, every local solution of \eqref{eq:main-pq-equation} is entire, and the
argument above already shows that $P g^{-2}\,dz$ has a meromorphic primitive
in $\mathbb C$. Thus, \textup{(ii)} implies \textup{(iii)}. The implication
\textup{(iii)}$\Rightarrow$\textup{(i)} is immediate.
\end{proof}

Under the assumptions of Theorem~\ref{thm:finite-zero-reduction}, the
meromorphic primitive $G$ is not an auxiliary object. Indeed, if
\[
f(z)=\sum_{n=1}^{\infty}\frac{c_n}{(z-t_n)^2},
\]
then, up to an additive constant,
\[
G(z)
=
-\sum_{n=1}^{\infty}
\left(\frac{c_n}{z-t_n}+\frac{c_n}{t_n}\right).
\]
Thus, the condition that $P/g^2$ have a meromorphic primitive is intrinsic to
the original problem: it is exactly the condition that the quadratic
Cauchy-kernel expansion be the derivative of a meromorphic sum of simple
Cauchy kernels.

\subsection{The decomposition criterion}

We now combine the Schwarzian reformulation with the sectorial criterion in
Theorem~\ref{thm:razloz_v_sum_1}. We write
$\mathbb P^1$ for the complex projective line.

\begin{theorem}\label{thm:decomposition-from-laine}
Let $P$ and $Q$ be as in Theorem~\ref{thm:zeros-stokes-localization}, and
assume that $2R$, with $R$ defined by
\eqref{eq:normal-potential-R}, satisfies the Laine condition of
Theorem~\ref{thm:laine}. Then every local solution of
\eqref{eq:main-pq-equation} extends to an entire function.

Let $g_1,g_2$ be a fundamental system of solutions of
\eqref{eq:main-pq-equation}. For $[\alpha:\beta]\in\mathbb P^1$, put
\[
g_{\alpha,\beta}=\alpha g_1+\beta g_2.
\]
Then there exists a finite set $E\subset\mathbb P^1$ such that, for every
$[\alpha:\beta]\notin E$, the function $g_{\alpha,\beta}$ tends to
$\infty$ in every Stokes sector. Consequently, if $g_{\alpha,\beta}$ and
$P$ have no common zeros, then
\[
\frac{P(z)}{g_{\alpha,\beta}(z)^2}
=
\sum_{n=1}^{\infty}\frac{c_n}{(z-t_n)^2},
\]
where $\{t_n\}$ is the zero set of $g_{\alpha,\beta}$.
\end{theorem}

\begin{proof}
By Theorem~\ref{thm:laine-entire-solutions}, every solution of
\eqref{eq:main-pq-equation} is entire. Fix a Stokes sector $S_k$. By
Theorem~\ref{thm:sectorial-behavior}, the solutions which tend to $0$ in
$S_k$ form a one-dimensional subspace of the solution space. Hence the
condition
\[
g_{\alpha,\beta}(z)\to0,
\qquad z\to\infty,\ z\in S_k,
\]
defines a single point $E_k\in\mathbb P^1$. Since there are only finitely
many Stokes sectors,
\[
E:=\bigcup_k E_k
\]
is finite. If $[\alpha:\beta]\notin E$, then $g_{\alpha,\beta}$ does not
tend to $0$ in any Stokes sector. By Theorem~\ref{thm:sectorial-behavior},
it follows that
\[
g_{\alpha,\beta}(z)\to\infty
\]
in every Stokes sector.

Assume in addition that $g_{\alpha,\beta}$ and $P$ have no common zeros.
Then $g_{\alpha,\beta}$ satisfies condition \textup{(ii)} of
Theorem~\ref{thm:razloz_v_sum_1}. The desired expansion follows from that
theorem.
\end{proof}

\section{Algebraic parametrization}

The results of the previous section reduce the existence problem to an
algebraic one. For fixed $P$ and fixed order $\rho$, the relevant
functions are determined by polynomials $Q$ for which the corresponding
rational function $2R$ satisfies the Laine condition. In this section we
show that these $Q$ form a finite-dimensional algebraic variety and that,
modulo the natural equivalence relation on the function side, this variety
parametrizes the corresponding class of meromorphic functions.

For a polynomial $P$ and $\rho\ge1$, we write $K_\rho(P)$ for the
class of functions of the form
\[
f(z)=\sum_{k=1}^{\infty}\frac{c_k}{(z-t_k)^2},
\qquad
\sum_{k=1}^{\infty}\frac{|c_k|}{|t_k|^2}<\infty,
\]
such that the pole sequence $\{t_k\}$ has order $\rho$, and the zero set
of $f$ coincides with the zero set of $P$, counted with multiplicities.

\subsection{The Laine variety}

Let
\[
P(z)=\prod_{k=1}^{n}(z-p_k)^{\alpha_k},
\qquad
\deg P=p,
\]
and fix $q\ge p$. We denote by $\mathcal L_q(P)$ the set of all
polynomials $Q$ of degree $q$ such that
\[
\frac{Q(z)}{P(z)}
=
S(z)+\sum_{k=1}^{n}\frac{D_k}{z-p_k},
\]
where $S$ is a polynomial of degree $q-p$, and the function
\[
2R
=
\left(\frac{P'}{P}\right)'
-
\frac12\left(\frac{P'}{P}\right)^2
+
\frac{2Q}{P}
\]
satisfies the Laine condition of Theorem~\ref{thm:laine} at all its poles. We shall call $\mathcal L_q(P)$ the
\emph{Laine variety} associated with $P$ and $q$. In the trivial case
$P\equiv1$, one has simply
\[
\mathcal L_q(1)=\{Q\in\mathbb C_q[z]:\deg Q=q\}.
\]

We shall use the following elementary lemma.

\begin{lemma}\label{lem:finite-algebraic-system}
Let $m_1,\dots,m_n\in\mathbb N$, and let
\[
F_k(D_1,\dots,D_n)=D_k^{m_k}+W_k(D_1,\dots,D_n),
\qquad k=1,\dots,n,
\]
where each $W_k\in\mathbb C[D_1,\dots,D_n]$ satisfies
\[
\deg W_k\le m_k-1.
\]
Then the system
\[
F_1(D)=\cdots=F_n(D)=0
\]
has only finitely many solutions in $\mathbb C^n$, and the sum of their
multiplicities is equal to $m_1\cdots m_n$.
\end{lemma}

\begin{proof}
Set
\[
F(D)=\bigl(F_1(D),\dots,F_n(D)\bigr),
\qquad
F^{(t)}(D)=
\bigl(D_1^{m_1}+tW_1(D),\dots,D_n^{m_n}+tW_n(D)\bigr),
\quad 0\le t\le1.
\]
The zeros of $F^{(t)}$, $0\le t\le1$, remain in a fixed bounded set.
Indeed, otherwise there would be sequences $D^{(j)}\in\mathbb C^n$ and
$t_j\in[0,1]$ such that
\[
F^{(t_j)}(D^{(j)})=0,
\qquad |D^{(j)}|\to\infty.
\]
Passing to a subsequence, we may assume that
\[
t_j\to t_0,
\qquad
\frac{D^{(j)}}{|D^{(j)}|}\to \xi,
\qquad |\xi|=1.
\]
Choosing $k$ with $\xi_k\neq0$, and dividing the $k$-th equation by
$|D^{(j)}|^{m_k}$, we get
\[
0=
\left(\frac{D_k^{(j)}}{|D^{(j)}|}\right)^{m_k}
+
t_j\,\frac{W_k(D^{(j)})}{|D^{(j)}|^{m_k}}
\longrightarrow
\xi_k^{m_k},
\]
since $\deg W_k\le m_k-1$, a contradiction.

Hence, for some sufficiently large $R$, the maps $F^{(t)}$ have no zeros
on $\partial B_R$. By the homotopy invariance of the Brouwer degree,
\[
\deg(F,B_R,0)=\deg(F^{(0)},B_R,0);
\]
see Deimling \cite[Ch.~1]{DeimlingNFA}. Since
\[
F^{(0)}(D)=(D_1^{m_1},\dots,D_n^{m_n}),
\]
we have
\[
\deg(F^{(0)},B_R,0)=m_1\cdots m_n.
\]
Thus
\[
\deg(F,B_R,0)=m_1\cdots m_n.
\]
The zero set $F^{-1}(0)$ is finite by the boundedness argument above. The
Brouwer degree is the sum of the local degrees, and, for a holomorphic map,
these local degrees are the algebraic multiplicities of the isolated zeros;
see again \cite[Ch.~1]{DeimlingNFA}. Therefore the sum of the multiplicities
of the solutions is $m_1\cdots m_n$.
\end{proof}

\begin{lemma}\label{lem:laine-determinant-leading-term}
Let $\mathcal D_m(b_1,\dots,b_m)$ be the determinant in
Theorem~\ref{thm:laine}. Then
\[
\mathcal D_m(b_1,\dots,b_m)
=
(-1)^{m+1}b_1^m+\widetilde{\mathcal D}_m(b_1,\dots,b_m),
\qquad
\deg \widetilde{\mathcal D}_m\le m-1,
\]
where $\deg$ is total degree in
$\mathbb C[b_1,\dots,b_m]$.
\end{lemma}

\begin{proof}
Expand $\mathcal D_m$ along the first row. The term coming from $1-m$
has degree at most $m-1$. The term coming from the entry $b_1$ in the last
column is
\[
(-1)^{m+1}b_1 H_{m-1},
\]
where
\[
H_{m-1}=
\det
\begin{pmatrix}
b_1 & c_1 & 0 & \cdots & 0\\
b_2 & b_1 & c_2 & \cdots & 0\\
b_3 & b_2 & b_1 & \ddots & \vdots\\
\vdots & \vdots & \vdots & \ddots & c_{m-2}\\
b_{m-1} & b_{m-2} & b_{m-3} & \cdots & b_1
\end{pmatrix}
\]
and $c_j=(j+1)^2-(j+1)m$.

It remains only to note that such determinants satisfy
\[
H_s=b_1^s+\widetilde H_s,
\qquad
\deg \widetilde H_s\le s-1.
\]
This follows by induction on $s$: expanding $H_s$ along the first row gives
\[
H_s=b_1H_{s-1}-c_1H_s',
\]
where $H_s'$ has size $s-1$ and hence degree at most $s-1$. Therefore
$b_1H_{s-1}$ contributes the unique term $b_1^s$ of degree $s$, while all
remaining terms have degree at most $s-1$.

Thus
\[
(-1)^{m+1}b_1H_{m-1}
=
(-1)^{m+1}b_1^m
+
\text{terms of degree at most }m-1,
\]
and the lemma follows.
\end{proof}

We may now compute the dimension of the Laine variety.

\begin{theorem}\label{thm:laine-variety}
For a fixed polynomial $P$ of degree $p$ and each $q\ge p$, the set
$\mathcal L_q(P)$ carries a natural structure of a complex affine algebraic
variety of dimension $q-p+1$.
\end{theorem}

\begin{proof}
Let $p_1,\dots,p_n$ be the distinct zeros of $P$, and let
$\alpha_k$ be the multiplicity of $p_k$. For $Q\in\mathcal L_q(P)$ we
write
\[
\frac{Q(z)}{P(z)}
=
S(z)+\sum_{k=1}^{n}\frac{D_k}{z-p_k},
\]
where $S\in\mathbb C_{q-p}[z]$ and
$D=(D_1,\dots,D_n)\in\mathbb C^n$. Thus $Q$ is uniquely determined by
$(S,D)$, and $\mathcal L_q(P)$ is realized as an algebraic subset of
$\mathbb C_{q-p}[z]\times\mathbb C^n$.

By Lemma~\ref{lem:schwarzian-partial-fractions},
\[
2R=
\sum_{k=1}^{n}
\frac{-\alpha_k(\alpha_k+2)}{2(z-p_k)^2}
+
\sum_{k=1}^{n}\frac{2D_k-M_k}{z-p_k}
+
2S(z).
\]
Comparing the coefficient of $(z-p_k)^{-2}$ with the canonical form in
Theorem~\ref{thm:laine}, for fixed $k$, we obtain
\[
m_k=\alpha_k+1.
\]
Hence, at $p_k$,
\[
R(z)=
\frac{1-m_k^2}{4(z-p_k)^2}
+\frac{b_{k,1}}{z-p_k}
+b_{k,2}+b_{k,3}(z-p_k)+\cdots ,
\]
where
\[
b_{k,1}=D_k-\frac{M_k}{2}.
\]
The coefficients $b_{k,2},\dots,b_{k,m_k}$ are affine-linear functions of
the coefficients of $S$ and of the residues $D_1,\dots,D_n$.

By Lemma~\ref{lem:laine-determinant-leading-term}, the Laine determinant has
the form
\[
\mathcal D_m(b_1,\dots,b_m)
=
(\pm1)b_1^m+\widetilde{\mathcal D}_m(b_1,\dots,b_m),
\qquad
\deg \widetilde{\mathcal D}_m\le m-1.
\]
Applying this with $m=m_k$, the local Laine equation at $p_k$, after
substitution of the coefficients $b_{k,j}$ and division by a non-zero
constant, has the form
\[
D_k^{m_k}+W_k(S,D_1,\dots,D_n)=0,
\qquad m_k=\alpha_k+1,
\]
where
\[
\deg_D W_k\le m_k-1=\alpha_k.
\]
Here $\deg_D$ denotes the total degree in the residue variables
$D_1,\dots,D_n$, while the coefficients of $W_k$ are polynomial functions
of the coefficients of $S$.

Thus, for each fixed $S$, the Laine condition gives a system of algebraic
equations in $D_1,\dots,D_n$ of the type covered by
Lemma~\ref{lem:finite-algebraic-system}. In particular, for each fixed $S$
this system has finitely many solutions, counted with multiplicity. Hence the
projection
\[
\pi:\mathcal L_q(P)\to\mathbb C_{q-p}[z],
\qquad
\pi(Q)=S,
\]
is a finite algebraic map. It follows that $\mathcal L_q(P)$ is an affine
algebraic variety of the same dimension as the base, namely
\[
\dim \mathcal L_q(P)=q-p+1.
\]
\end{proof}

\subsection{The natural equivalence relation}

We next describe the natural equivalence relation on the function side. Every
$f\in K_\rho(P)$ has a meromorphic primitive. If $G'=f$ and
$T$ is a fractional-linear transformation, then
\[
(T\circ G)'=T'(G)f.
\]
Thus, the natural operation is not on $f$ itself, but on its primitive.

We write $f_1\sim f_2$ if there exist meromorphic primitives $G_1,G_2$ of
$f_1,f_2$, respectively, and a fractional-linear transformation $T$ such
that
\[
G_2=T\circ G_1.
\]
Equivalently, if
\[
T(w)=\frac{aw+b}{cw+d},
\qquad ad-bc\neq0,
\]
then, for a meromorphic primitive $G$ of $f_1$,
\[
f_2(z)=
\frac{(ad-bc)f_1(z)}{(cG(z)+d)^2}.
\]
Changing the additive constant in $G$ merely changes $T$, and hence the
definition is independent of the chosen primitive. It is immediate from the
definition that $\sim$ is an equivalence relation on $K_\rho(P)$.

\subsection{The moduli interpretation}

We can now identify the quotient $K_\rho(P)/{\sim}$ with the
corresponding Laine variety.

\begin{theorem}\label{thm:moduli-bijection}
Let $[f]\in K_\rho(P)/{\sim}$, and let $Q$ be the polynomial determined
by
\[
\mathcal A(f)=\mathcal A(P)+\frac{2Q}{P}.
\]
Then the map
\[
\pi : K_\rho(P)/{\sim} \to \mathcal L_{2\rho+p-2}(P),
\qquad [f]\mapsto Q,
\]
is a bijection. In particular, the quotient $K_\rho(P)/{\sim}$ carries
a natural structure of a complex affine algebraic variety of dimension
$2\rho-1$.
\end{theorem}

\begin{proof}
Let $f\in K_\rho(P)$, and let $G$ be a meromorphic primitive of $f$.
By Theorem~\ref{thm:finite-zero-reduction}, $f=\frac{P}{g^2}$, where $g$ is an entire solution of \eqref{eq:osnovnoe}, and the associated
polynomial $Q$ satisfies
\[
\rho=\frac{\deg Q-\deg P+2}{2}.
\]
Since $G'=f=P/g^2$, Theorem~\ref{thm:laine-entire-solutions} implies that
$2R$ satisfies the Laine condition. Hence $Q\in\mathcal L_{2\rho+p-2}(P)$.

We next check that $\pi$ is defined on equivalence classes. Suppose that
$f_1\sim f_2$. Choose primitives $G_1,G_2$ with
\[
G_2=T\circ G_1
\]
for some fractional-linear transformation $T$. Since the Schwarzian
derivative is invariant under postcomposition by fractional-linear maps,
\[
\mathcal A(f_2)=\mathcal S(G_2)=\mathcal S(G_1)=\mathcal A(f_1).
\]
Thus, the polynomial $Q$ associated with $f$ depends only on equivalence
classes.

Conversely, assume that $\pi([f_1])=\pi([f_2])=Q$, and let $G_j$ be a
meromorphic primitive of $f_j$, $j=1,2$. Then
\[
\mathcal S(G_j)=\mathcal A(f_j)=\mathcal A(P)+\frac{2Q}{P},
\qquad j=1,2.
\]
Hence $G_1$ and $G_2$ have the same Schwarzian derivative. Therefore they
differ by a fractional-linear transformation, and so $f_1\sim f_2$. This
proves injectivity.

It remains to prove surjectivity. Let
\[
Q\in\mathcal L_{2\rho+p-2}(P).
\]
Then the corresponding rational function $2R$ satisfies the Laine
condition, and by Theorem~\ref{thm:laine-entire-solutions} all solutions of
\eqref{eq:osnovnoe} are entire. By
Theorem~\ref{thm:decomposition-from-laine}, we may choose a solution $g$
with no common zeros with $P$ such that
\[
f=\frac{P}{g^2}\in K_\rho(P).
\]
For this function one has
\[
\mathcal A(f)=\mathcal A(P)+\frac{2Q}{P},
\]
and hence $\pi([f])=Q$. Thus, $\pi$ is surjective.

Consequently $\pi$ is a bijection. Finally,
Theorem~\ref{thm:laine-variety} gives
\[
\dim \mathcal L_{2\rho+p-2}(P)
=
(2\rho+p-2)-p+1
=
2\rho-1,
\]
and the same dimension formula holds for $K_\rho(P)/{\sim}$.
\end{proof}

As a concrete consequence, we recover the corresponding existence statement
in the original language of sums of squares of Cauchy kernels.

\begin{theorem}\label{thm:existence-given-P-order}
Let $P$ be a polynomial of degree $p\in\mathbb N_0$, and let
$N\in\mathbb N_0$. Then there exist sequences
$\{c_k\},\{t_k\}\subset\mathbb C$, with $|t_k|\to\infty$, such that
\[
f(z)=\sum_{k=1}^{\infty}\frac{c_k}{(z-t_k)^2},
\qquad
\sum_{k=1}^{\infty}\frac{|c_k|}{|t_k|^2}<\infty,
\]
has order
\[
\rho(f)=\frac{N+2}{2},
\]
and its zero set, counted with multiplicities, coincides with the zero set
of $P$.
\end{theorem}

\begin{proof}
By Theorem~\ref{thm:laine-variety}, the variety
$\mathcal L_{N+p}(P)$ is nonempty. Hence, by
Theorem~\ref{thm:moduli-bijection}, the class $K_{\frac{N+2}{2}}(P)$ is nonempty,
which is exactly the required statement.
\end{proof}

\section{Examples and explicit model families}

We now discuss several explicit model families, generated by various polynomial $P$. Recall, that we studied above entire solutions $g$ of 
\begin{equation}\label{eq:osnov_eq_for_examples}
    g''-\frac{P'}{P}g'+\frac{Q}{P}g=0
\end{equation}

\subsection{Monomial models: Bessel and Airy families}

\subsubsection*{Bessel families for $P(z)=z^p$}

We use standard facts about Bessel functions; see
\cite[Ch.~10]{Olver_NIST} and \cite{NIST_DLMF}. The functions $J_\alpha$ and $Y_\alpha$ form a fundamental system of
solutions of
\[
z^2u''(z)+zu'(z)+(z^2-\alpha^2)u(z)=0.
\]
They are analytic in the slit plane $\mathbb C\setminus(-\infty,0]$, 
where the principal branch of $z^\alpha$ is fixed. Moreover,
\[
J_\alpha(z)=z^\alpha \Phi_\alpha(z),
\]
where $\Phi_\alpha$ is an entire function. If
$\alpha\notin\mathbb Z$, then
\[
Y_\alpha(z)=
\frac{J_\alpha(z)\cos \pi\alpha-J_{-\alpha}(z)}
{\sin \pi\alpha}.
\]
For $\alpha\in\mathbb Z$, the function $Y_\alpha$ is obtained by taking
the limit in the last formula; in this case logarithmic branching appears at
the origin. Finally, as $z\to\infty$ in any
sector $|\arg z|\le \pi-\varepsilon$,
\[
J_\alpha(z)
=
\sqrt{\frac{2}{\pi z}}
\left(
\cos\!\left(z-\frac{\pi\alpha}{2}-\frac{\pi}{4}\right)
+
O(z^{-1})
\right),
\]
\[
Y_\alpha(z)
=
\sqrt{\frac{2}{\pi z}}
\left(
\sin\!\left(z-\frac{\pi\alpha}{2}-\frac{\pi}{4}\right)
+
O(z^{-1})
\right).
\]

\begin{theorem}
Let $p\in\mathbb N_0$, and suppose that one of the following conditions
holds:
\begin{enumerate}
\item[\textup{(i)}] $p=0$, $N\in\mathbb N_0$, and
$(A,B)\in\mathbb C^2\setminus\{(0,0)\}$;
\item[\textup{(ii)}] $p\ge1$, $N\in\mathbb N$, the number $N+2$ does
not divide $p+1$, and $A\in\mathbb C$, $B\in\mathbb C\setminus\{0\}$.
\end{enumerate}
Then there exist sequences $\{c_n\},\{t_n\}\subset\mathbb C$,
$|t_n|\to\infty$, such that
\[
\frac{z^p}{g^2(z)}
=
\sum_{n=1}^{\infty}\frac{c_n}{(z-t_n)^2},
\qquad
\sum_{n=1}^{\infty}\frac{|c_n|}{|t_n|^2}<\infty,
\]
where
\[
g(z)
=
A\, z^{\frac{p+1}{2}}
J_{\frac{p+1}{N+2}}\!\left(a z^{\frac{N+2}{2}}\right)
+
B\, z^{\frac{p+1}{2}}
Y_{\frac{p+1}{N+2}}\!\left(a z^{\frac{N+2}{2}}\right)
\]
is an entire function of order $\frac{N+2}{2}$. In particular, when
$N=2p$,
\[
g(z)=A\sin\!\bigl(a z^{p+1}\bigr)+B\cos\!\bigl(a z^{p+1}\bigr).
\]
\end{theorem}

\begin{proof}
Set
\[
P(z)=z^p,
\qquad
Q(z)=\left(\frac{a(N+2)}{2}\right)^2 z^{p+N}.
\]
Then \eqref{eq:osnov_eq_for_examples} becomes
\[
g''(z)-\frac{p}{z}g'(z)
+\left(\frac{a(N+2)}{2}\right)^2 z^N g(z)=0.
\]
With the substitution
\[
g(z)=z^{\frac{p+1}{2}}u\!\left(a z^{\frac{N+2}{2}}\right),
\]
this reduces to the Bessel equation, and hence
\[
g(z)
=
A\,z^{\frac{p+1}{2}}
J_{\frac{p+1}{N+2}}\!\left(a z^{\frac{N+2}{2}}\right)
+
B\,z^{\frac{p+1}{2}}
Y_{\frac{p+1}{N+2}}\!\left(a z^{\frac{N+2}{2}}\right).
\]

It follows from the properties of $J_\alpha$ and $Y_\alpha$, together with
the fact that $\frac{p+1}{N+2}\notin\mathbb Z$, that $g$ is an entire
function and $g(0)\ne0$. Let $z$ belong to a fixed Stokes sector of the equation. Then
\[
w=a z^{\frac{N+2}{2}}
\]
lies in a sector $|\arg w|\le\pi-\varepsilon$, and the large-argument
asymptotics of $J_\alpha$ and $Y_\alpha$ yield
\[
g(z)=z^{\frac{2p-N}{4}}
\left(
A_1e^{iw}+B_1e^{-iw}
\right)\bigl(1+O(w^{-1})\bigr)
\]
for suitable constants $A_1,B_1$. Thus, in each Stokes sector one
exponential term is increasing and the other decreasing, exactly as predicted
by the general theory, and in particular $g(z)\to\infty$ in every Stokes
sector. The required decomposition therefore follows from
Theorem~\ref{thm:razloz_v_sum_1}.

If $N=2p$, then the order of the Bessel functions is $1/2$, and the
standard elementary identities reduce the solution to
\[
g(z)=A\sin\!\bigl(a z^{p+1}\bigr)+B\cos\!\bigl(a z^{p+1}\bigr).
\]
\end{proof}

\subsubsection*{Airy families and the order $\frac32$ case}

We use standard facts about Airy functions; see
\cite[Ch.~9]{Olver_NIST} and \cite{NIST_DLMF}. The functions $\Ai$ and $\Bi$ are linearly independent entire solutions of
\[
u''(z)-zu(z)=0,
\]
and both have order $3/2$. We shall use the standard Airy asymptotics. As
$z\to\infty$, uniformly in $|\arg z|\le \pi-\delta$,
\[
\Ai(z)=
\frac{1}{2\sqrt{\pi}}\,
z^{-1/4}\exp\!\left(-\frac23 z^{3/2}\right)(1+o(1)).
\]
For $\Bi$, the elementary exponential asymptotic
\[
\Bi(z)=
\frac{1}{\sqrt{\pi}}\,
z^{-1/4}\exp\!\left(\frac23 z^{3/2}\right)(1+o(1))
\]
holds uniformly in $|\arg z|\le \pi/3-\delta$. In the remaining sectors one
uses the standard connection formulae for Airy functions.

Thus $\Ai$ tends to $0$ in the Stokes sector containing the positive real
axis and tends to $\infty$ in the other two sectors. On the other hand,
$\Bi$ tends to $\infty$ in all three Stokes sectors, after the sectors are
shrunk away from their boundary rays.

\begin{theorem}
Let $\{c_n\},\{t_n\}\subset\mathbb C$ satisfy $|t_n|\to\infty$ and
\[
f(z)=\sum_{n=1}^{\infty}\frac{c_n}{(z-t_n)^2},
\qquad
\sum_{n=1}^{\infty}\frac{|c_n|}{|t_n|^2}<\infty.
\]
Assume that $f$ is transcendental, has no zeros, and $\rho(f)=3/2$.
Then
\[
f(z)=\frac{1}{g^2(z)},
\]
where
\[
g(z)=A\,\Ai(az+b)+B\,\Bi(az+b),
\qquad
a,B\in\mathbb C\setminus\{0\},\quad A,b\in\mathbb C.
\]
\end{theorem}

\begin{proof}
By Theorem~\ref{thm:finite-zero-reduction},
\[
f=\frac{1}{g^2},
\]
where $g$ is an entire solution of
\[
g''+Qg=0
\]
with a polynomial $Q$. Since
\[
\rho(f)=\rho(g)=\frac{\deg Q+2}{2}=\frac32,
\]
we have $\deg Q=1$, and therefore we can write
\[
Q(z)=-a^3z+ba^2,
\qquad
a\in\mathbb C\setminus\{0\},\ b\in\mathbb C.
\]
Hence
\[
g''(z)+\bigl(-a^3z+ba^2\bigr)g(z)=0,
\]
and an affine change of variables reduces this equation to the Airy equation.
Thus
\[
g(z)=A\,\Ai(az+b)+B\,\Bi(az+b).
\]

Since $f$ is zero-free, Theorem~\ref{thm:razloz_v_sum_1} implies that
$g(z)\to\infty$ in every Stokes sector. By the Airy asymptotics, $\Ai$
decays exponentially in one Stokes sector, whereas $\Bi$ grows in all
three. Hence $B=0$ is impossible, and therefore $B\neq0$. Consequently,
the condition of Theorem~\ref{thm:razloz_v_sum_1} is fulfilled, and the
required decomposition follows.
\end{proof}

\subsection{Kummer models: low-order and borderline cases}

We use standard facts about Kummer functions; see
\cite[Ch.~13]{Olver_NIST} and \cite{NIST_DLMF}. The Kummer function
$\mathrm M(\alpha,\beta,z)$ satisfies
\begin{equation}\label{eq:Kummer-equation}
z f''(z)+(\beta-z)f'(z)-\alpha f(z)=0.
\end{equation}
If $\beta\notin\mathbb Z$, then
\[
\mathrm M(\alpha,\beta,z),
\qquad
z^{1-\beta}\mathrm M(\alpha-\beta+1,2-\beta,z)
\]
form a fundamental system of solutions of \eqref{eq:Kummer-equation}. We shall use the standard asymptotic formula for Kummer's function:
for $|\arg z|\le \pi-\delta$,
\begin{equation}\label{eq:kummer_asymptotics}
\mathrm M(\alpha,\beta,z)
=
\frac{e^z z^{\alpha-\beta}}{\Gamma(\alpha)}
\bigl(1+O(z^{-1})\bigr)
+
\frac{e^{\pi i\alpha}z^{-\alpha}}{\Gamma(\beta-\alpha)}
\bigl(1+O(z^{-1})\bigr).
\end{equation}

\subsubsection*{Zero-free functions of order $2$}

\begin{theorem}
Let $\{c_n\},\{t_n\}\subset\mathbb C$ satisfy $|t_n|\to\infty$ and
\[
f(z)=\sum_{n=1}^{\infty}\frac{c_n}{(z-t_n)^2},
\qquad
\sum_{n=1}^{\infty}\frac{|c_n|}{|t_n|^2}<\infty.
\]
Assume that $f$ is transcendental, has order $2$, and has no zeros.
Then
\[
f(z)=\frac{1}{g^2(z)},
\]
where
\[
g(z)
=
e^{-a(z-b)^2/2}
\left(
A\,(z-b)\,
\mathrm M\!\left(c+\tfrac12,\tfrac32,a(z-b)^2\right)
+
B\,
\mathrm M\!\left(c,\tfrac12,a(z-b)^2\right)
\right),
\]
with $b,c\in\mathbb C$ and $a,A,B\in\mathbb C\setminus\{0\}$. In
particular, if $c=0$, then, after renaming the constants $A$ and $B$, we
obtain
\[
g(z)
=
e^{-a(z-b)^2/2}
\left(
B+
A\int\limits_0^{z-b} e^{a\zeta^2}\,d\zeta
\right).
\]
\end{theorem}

\begin{proof}
By Theorem~\ref{thm:finite-zero-reduction},
\[
f=\frac{1}{g^2},
\]
where $g$ is an entire solution of
\[
g''+Qg=0
\]
with a polynomial $Q$. Since
\[
\rho(f)=\rho(g)=\frac{\deg Q+2}{2}=2,
\]
we have $\deg Q=2$, and therefore we can write
\[
Q(z)=-a^2z^2+2a^2bz+a(1-4c-a b^2),
\qquad a\neq0.
\]
Hence
\[
g''(z)+\left(-a^2z^2+2a^2bz+a(1-4c-a b^2)\right)g(z)=0.
\]
With
\[
w=a(z-b)^2,
\qquad
g(z)=\sqrt w\,e^{-w/2}G(w),
\]
this reduces to
\[
wG''(w)+\left(\frac32-w\right)G'(w)-\left(c+\frac12\right)G(w)=0,
\]
that is, to a Kummer equation. Therefore
\[
G(w)=
A\,\mathrm M\!\left(c+\tfrac12,\tfrac32,w\right)
+
B\,w^{-1/2}\mathrm M\!\left(c,\tfrac12,w\right),
\]
and hence
\begin{equation}\label{eq:kummer_formula_sol}
g(z)
=
e^{-a(z-b)^2/2}
\left(
A\,(z-b)\,
\mathrm M\!\left(c+\tfrac12,\tfrac32,a(z-b)^2\right)
+
B\,
\mathrm M\!\left(c,\tfrac12,a(z-b)^2\right)
\right).
\end{equation}

By the large-$w$ asymptotics of $\mathrm M(\alpha,\beta,w)$, see
\eqref{eq:kummer_asymptotics}, each Kummer term is a sum of an $e^w$-part
and an algebraic part; after multiplication by
$e^{-w/2}$, these become the two opposite exponential behaviors
$e^{w/2}$ and $e^{-w/2}$, exactly as predicted by the general theory. In the resonant cases, namely when
$\alpha\in-\mathbb N_0$ or $\beta-\alpha\in-\mathbb N_0$, the Kummer
functions reduce to Hermite-polynomial solutions, and the same conclusion
follows. Since $f$ is zero-free, Theorem~\ref{thm:razloz_v_sum_1} implies that
$g(z)\to\infty$ in every Stokes sector. Consequently both sectorial
contributions must be present, and therefore $A\neq0$ and $B\neq0$.
The conclusion now follows as in the previous case. If $c=0$, then, using
\[
\mathrm M\!\left(0,\frac12,w\right)=\frac1{\sqrt\pi},
\qquad
x\,\mathrm M\!\left(\frac12,\frac32,a x^2\right)
=
\frac{2}{\sqrt\pi}
\int\limits_0^x e^{a\zeta^2}\,d\zeta ,
\]
and renaming the constants $A$ and $B$, formula
\eqref{eq:kummer_formula_sol} takes the form
\[
g(z)
=
e^{-a(z-b)^2/2}
\left(
B+
A\int\limits_0^{z-b} e^{a\zeta^2}\,d\zeta
\right).
\]
\end{proof}

\subsubsection*{The case $P(z)=z^p$ and order $1$}

We shall use the following explicit expansion.

\begin{lemma}\label{lem:kummer-expansion-order-one}
Let $0\le k\le p$, and put
\[
C_{p,k}=\frac{1}{k!(p-k)!}.
\]
Then
\[
\mathrm M(k+1,p+2,z)
=
e^z\sum_{r=p-k+1}^{p+1}\frac{\gamma_r}{z^r}
+
\sum_{r=k+1}^{p+1}\frac{\delta_r}{z^r},
\]
where
\[
\gamma_r
=
C_{p,k}
(-1)^{r-p+k-1}
(r-1)!
\binom{k}{r-p+k-1},
\qquad p-k+1\le r\le p+1,
\]
and
\[
\delta_r
=
C_{p,k}
(-1)^{k+1}
(r-1)!
\binom{p-k}{r-k-1},
\qquad k+1\le r\le p+1.
\]
In particular, both sums are nonzero.
\end{lemma}

\begin{proof}
We use the Euler representation
\[
\mathrm M(k+1,p+2,z)
=
C_{p,k}\int\limits_0^1 e^{zt}t^k(1-t)^{p-k}\,dt.
\]
Set
\[
\phi(t)=t^k(1-t)^{p-k}.
\]
Since $\phi$ is a polynomial of degree $p$, repeated integration by parts
gives the finite identity
\[
\int\limits_0^1 e^{zt}\phi(t)\,dt
=
e^z\sum_{r=1}^{p+1}
\frac{(-1)^{r-1}\phi^{(r-1)}(1)}{z^r}
-
\sum_{r=1}^{p+1}
\frac{(-1)^{r-1}\phi^{(r-1)}(0)}{z^r}.
\]
At $t=1$, the first nonzero derivative has order $p-k$, and for
$p-k+1\le r\le p+1$ one has
\[
(-1)^{r-1}\phi^{(r-1)}(1)
=
(-1)^{r-p+k-1}
(r-1)!
\binom{k}{r-p+k-1}.
\]
At $t=0$, the first nonzero derivative has order $k$, and for
$k+1\le r\le p+1$ one has
\[
-(-1)^{r-1}\phi^{(r-1)}(0)
=
(-1)^{k+1}
(r-1)!
\binom{p-k}{r-k-1}.
\]
Multiplying by $C_{p,k}$ gives the stated formula.
\end{proof}

Define the polynomials
\[
\Gamma_{p,k}(w)
=
\sum_{r=0}^{k}\gamma_{p+1-r}w^r,
\qquad
\Delta_{p,k}(w)
=
\sum_{r=0}^{p-k}\delta_{p+1-r}w^r,
\]
and
\[
H_{p,k}(w)
=
\sum_{r=0}^{p-k}h_{p+1-r}w^r,
\qquad
h_r=\frac{(r-1)!}{(p+1-r)!(r-k-1)!}.
\]

\begin{theorem}\label{thm:monomial-order-one}
Let $\{c_n\},\{t_n\}\subset\mathbb C$ satisfy $|t_n|\to\infty$ and
\[
f(z)=\sum_{n=1}^{\infty}\frac{c_n}{(z-t_n)^2},
\qquad
\sum_{n=1}^{\infty}\frac{|c_n|}{|t_n|^2}<\infty.
\]
Assume that $f$ is transcendental, has order $1$, and has a unique zero
of multiplicity $p$ at the origin. Then
\[
f(z)=\frac{z^p}{g(z)^2},
\]
where, for some $a,A,B\in\mathbb C\setminus\{0\}$ and some
$k\in\{0,\dots,p\}$,
\begin{equation}\label{eq:example_z^p_formula}
g(z)
=
A e^{az}\Gamma_{p,k}(2az)
+
e^{-az}\bigl(A\Delta_{p,k}(2az)+B H_{p,k}(2az)\bigr).
\end{equation}

\end{theorem}

\begin{proof}
By Theorem~\ref{thm:finite-zero-reduction} $f=\frac{P}{g^2}$, where $P(z)=z^p$ and
 $g$ is an entire solution of equation \eqref{eq:osnov_eq_for_examples}.
Since
\[
\rho(f)=\rho(g)=\frac{\deg Q-\deg P+2}{2}=1,
\]
we have $\deg Q=p$. It follows from equation \eqref{eq:osnov_eq_for_examples} that $z^{p-1}$ divides
$Q(z)$ and thus
\[
        Q(z)=-a^2z^p+bz^{p-1},
        \qquad a\neq0.
\]
and $g$ satisfies
\begin{equation*}\label{eq:monomial-order-one-ode}
g''(z)-\frac{p}{z}g'(z)+\left(-a^2+\frac{b}{z}\right)g(z)=0.
\end{equation*}

By Theorem~\ref{thm:laine-entire-solutions}, the corresponding potential
satisfies the Laine condition. Here
\[
\frac{P'}{P}=\frac{p}{z},
\qquad
\frac{Q}{P}=-a^2+\frac{b}{z},
\]
and hence
\[
2R(z)
=
-\frac{p(p+2)}{2z^2}
+\frac{2b}{z}
-2a^2 .
\]
Thus, the only pole of $2R$ is at the origin, with exponent difference
$p+1$, and the Laine condition reduces to the single equation
$D_{p+1}=0$. In the present case $D_{p+1}$ is, up to a non-zero numerical factor, the
tridiagonal determinant
\[
        \Delta_p(a,b)=\det M_p(a,b),
\]
where
\[
M_p(a,b)=
\begin{pmatrix}
b & -p & 0 & \cdots & 0\\
-a^2 & b & -2(p-1) & \ddots & \vdots\\
0 & -a^2 & b & \ddots & 0\\
\vdots & \ddots & \ddots & \ddots & -p\\
0 & \cdots & 0 & -a^2 & b
\end{pmatrix}
\]
and $(M_p)_{j,j+1}=-(j+1)(p-j)$.
Consider the diagonal matrix
\[
        \Lambda_a=
        \operatorname{diag}
        \left(1,\frac{1!}{a},\frac{2!}{a^2},\dots,\frac{p!}{a^p}\right).
\]
Then direct calculation gives
\[
        M_p(a,b)=\Lambda_a^{-1}(bI-aT_p)\Lambda_a,
\]
where
\[
T_p=
\begin{pmatrix}
0 & p & 0 & \cdots & 0\\
1 & 0 & p-1 & \cdots & 0\\
0 & 2 & 0 & \ddots & 0\\
\vdots & \vdots & \ddots & \ddots & 1\\
0 & 0 & \cdots & p & 0
\end{pmatrix}
\]
is the transpose of the Sylvester--Kac matrix. Hence
\[
        \Delta_p(a,b)=\det(bI-aT_p).
\]
By Sylvester's determinant formula, as recalled by Askey
\cite[Sec.~1]{Askey2005}, the eigenvalues of $T_p$ are $p-2k$,
$k=0,\dots,p$. Therefore
\[
        \Delta_p(a,b)
        =
        \prod_{k=0}^{p}\bigl(b-a(p-2k)\bigr).
\]
Since $D_{p+1}$ differs from $\Delta_p(a,b)$ by a non-zero factor, the
Laine condition gives 
\[
b=(p-2k)a
\]
for some $k\in\{0,\dots,p\}$. With this value of $b$, set
\[
w=2az,
\qquad
g(z)=G(w)\,w^{p+1}e^{-w/2}.
\]
Then \eqref{eq:monomial-order-one-ode} becomes
\[
wG''(w)+(p+2-w)G'(w)-(k+1)G(w)=0,
\]
that is, the Kummer equation with parameters
\[
\alpha=k+1,
\qquad
\beta=p+2.
\]
By the standard theory of Kummer's equation in the resonant case
$\beta\in\mathbb N$, see \cite[Ch.~13]{Olver_NIST} and
\cite[\S13.2]{NIST_DLMF}, a fundamental system is given by
\[
\mathrm M(k+1,p+2,w),
\qquad
\Psi_{k,p}(w),
\]
where the second solution is the rational solution
\[
\Psi_{k,p}(w)=\sum_{r=k+1}^{p+1}\frac{h_r}{w^r},
\qquad
h_r=\frac{(r-1)!}{(p+1-r)!(r-k-1)!}.
\]
Hence
\[
G(w)=A\,\mathrm M(k+1,p+2,w)+B\,\Psi_{k,p}(w).
\]
By Lemma~\ref{lem:kummer-expansion-order-one},
\[
\mathrm M(k+1,p+2,w)
=
e^w\sum_{r=p-k+1}^{p+1}\frac{\gamma_r}{w^r}
+
\sum_{r=k+1}^{p+1}\frac{\delta_r}{w^r}.
\]
Hence
\[
w^{p+1}e^{-w/2}\mathrm M(k+1,p+2,w)
=
e^{w/2}\Gamma_{p,k}(w)
+
e^{-w/2}\Delta_{p,k}(w).
\]
Moreover,
\[
w^{p+1}e^{-w/2}\Psi_{k,p}(w)
=
e^{-w/2}H_{p,k}(w).
\]
Since $w=2az$, the asserted formula for $g$ follows.

It remains only to exclude the degenerate coefficients. Since $g$ and
$P(z)=z^p$ have no common zeros, one has $g(0)\neq0$. From the displayed
formula \eqref{eq:example_z^p_formula},
\[
g(0)=2aB\binom{p}{k}k!,
\]
and therefore $B\neq0$. On the other hand,
Lemma~\ref{lem:kummer-expansion-order-one} shows that
$\mathrm M(k+1,p+2,w)$ contains both an $e^w$-term and a rational term;
after multiplication by $w^{p+1}e^{-w/2}$, these give the two opposite
exponential behaviors $e^{az}$ and $e^{-az}$. Since Theorem~\ref{thm:razloz_v_sum_1} implies that $g(z)\to\infty$ in
every Stokes sector, both exponential contributions must be present. Hence
$A\neq0$, and the required decomposition follows.
\end{proof}

\subsubsection*{The borderline case $P(z)=z$}

This example realizes the borderline phenomenon from
Theorem~\ref{thm:zeros-stokes-localization}: the poles lie in logarithmic
neighborhoods of the Stokes rays but, in general, do not converge to those
rays in the Euclidean metric.

We shall also use the Lambert $W$-function, the multivalued inverse of
$w\mapsto we^w$. Its branches $W_k$, $k\in\mathbb Z$, are determined by
\[
W_k(z)e^{W_k(z)}=z.
\]
We refer to \cite[Ch.~4]{Olver_NIST} and \cite[\S4.13]{NIST_DLMF}. For each
fixed $\zeta\in\mathbb C\setminus\{0\}$,
\[
W_k(\zeta)
=
\log \zeta+2\pi i k
-
\log\!\bigl(\log \zeta+2\pi i k\bigr)
+o(1),
\qquad |k|\to\infty,
\]
with the logarithms taken on the corresponding branch. In particular,
\[
W_k(\zeta)=2\pi i k-\log(2\pi |k|)+O(1),
\qquad |k|\to\infty .
\]

\begin{theorem}\label{thm:borderline-order-one-simple-zero}
Let $\{c_n\},\{t_n\}\subset\mathbb C$ be sequences such that
$|t_n|\to\infty$ and
\[
f(z)=\sum_{n=1}^{\infty}\frac{c_n}{(z-t_n)^2},
\qquad
\sum_{n=1}^{\infty}\frac{|c_n|}{|t_n|^2}<\infty.
\]
Assume that $f$ is transcendental, has order $1$, and has a unique zero
of multiplicity $1$ at the point $0$. Then
\[
f(z)=\frac{z}{g^2(z)},
\]
where, for some $a,A,B\in\mathbb C\setminus\{0\}$, the function $g$ has
one of the forms
\[
g(z)=A e^{az}+B e^{-az}(2az+1),
\qquad
g(z)=A e^{az}(2az-1)+B e^{-az}.
\]
Moreover, the poles of $f$ lie in logarithmic neighborhoods of the Stokes
rays; in general, they do not approach those rays in the Euclidean metric.
\end{theorem}

\begin{proof}
By Theorem~\ref{thm:monomial-order-one}, applied with $p=1$, we have
$P(z)=z$. The two possible values $k=0$ and $k=1$ give respectively
\[
g(z)=A e^{az}+B e^{-az}(2az+1)
\]
and
\[
g(z)=A e^{az}(2az-1)+B e^{-az}.
\]

It remains to describe the zero distribution in this borderline case. We treat
the first representation of $g$; the second one is handled in the same way.
In this case the zeros of $g$ are precisely the solutions of
\[
        A e^{az}+B e^{-az}(2az+1)=0.
\]
Equivalently,
\[
(2az+1)e^{-2az}=-\frac{A}{B},
\]
and therefore
\[
z_k=-\frac{1}{2a}-\frac{1}{2a}\,
W_k\!\left(\frac{A}{eB}\right),
\qquad k\in\mathbb Z,
\]
where $W_k$ denotes the $k$-th branch of the Lambert $W$-function. Since
\[
W_k(\zeta)=2\pi i k-\log |k|+O(1),
\qquad |k|\to\infty,
\]
for fixed $\zeta\neq0$, we obtain
\[
z_k=-\frac{\pi i k}{a}
+\frac{1}{2a}\log |k|+O(1),
\qquad |k|\to\infty.
\]
Thus, the zeros lie at logarithmic distance from the corresponding Stokes
rays. The logarithmic term is nonconstant; for generic values of $A/B$ it
prevents Euclidean convergence to the rays.
\end{proof}

\subsection{A two-point example: $P(z)=z(z-1)$}

This is the first nontrivial example with two prescribed zeros. Even in this
case the Laine condition leads to a finite explicit list of admissible model
families.

\begin{theorem}\label{thm:two-point-order-one}
Let $\{c_n\},\{t_n\}\subset\mathbb C$ satisfy $|t_n|\to\infty$ and
\[
f(z)=\sum_{n=1}^{\infty}\frac{c_n}{(z-t_n)^2},
\qquad
\sum_{n=1}^{\infty}\frac{|c_n|}{|t_n|^2}<\infty.
\]
Assume that $f$ is transcendental, has order $1$, and has exactly two
simple zeros, at $0$ and $1$. Then
\[
f(z)=\frac{z(z-1)}{g^2(z)},
\]
where $g$ is entire, and there exist $a,A,B\in\mathbb C\setminus\{0\}$
such that
\[
g(z)=A e^{az}P_1(z)+B e^{-az}P_2(z),
\]
where the pair $(P_1,P_2)$ is one of the following:
\[
\begin{array}{c|c}
P_1(z) & P_2(z)\\[0.4ex]
\hline
2a^2z^2-2a(a+1)z+a+1 & 1\\[0.6ex]
1 & 2a^2z^2-2a(a-1)z-a+1\\[0.6ex]
z(1+a+\sqrt{a^2+1})-1
&
z(1-a+\sqrt{a^2+1})-1\\[0.6ex]
z(1+a-\sqrt{a^2+1})-1
&
z(a-1+\sqrt{a^2+1})+1
\end{array}
\]
The parameters are chosen so that $g$ and $z(z-1)$ have no common zeros.
\end{theorem}

\begin{proof}
By Theorem~\ref{thm:finite-zero-reduction} $f=\frac{P}{g^2}$ where $P(z)=z(z-1)$
and $g$ is an entire solution of equation \eqref{eq:osnov_eq_for_examples}. Since
\[
\rho(g)=\frac{\deg Q-\deg P+2}{2}=1,
\]
we have $\deg Q=2$. Thus, for some $a,b,c$
\[
Q(z)=-a^2z^2+bz+c,
\qquad a\neq0.
\]

By Theorem~\ref{thm:laine-entire-solutions}, the corresponding potential
satisfies the Laine condition. In the present case,
\[
R(z)
=
-\frac{3}{4z^2}
+
\frac{\frac12-c}{z}
-
\left(\frac14+b+c\right)
+O(z),
\qquad z\to0,
\]
and
\[
R(z)
=
-\frac{3}{4(z-1)^2}
+
\frac{-\frac12-a^2+b+c}{z-1}
-
\left(\frac14+a^2+c\right)
+O(z-1),
\qquad z\to1.
\]
Both poles have Laine parameters $m=2$. Hence the local Laine
condition is $D_2=0$, that is,
\[
\beta_1^2+\beta_2=0
\]
for the expansion
\[
R(z)=
-\frac{3}{4(z-z_0)^2}
+
\frac{\beta_1}{z-z_0}
+\beta_2+O(z-z_0).
\]
at each of poles. Applying this at $0$ and $1$, we get
\[
\left(\tfrac12-c\right)^2-\tfrac14-b-c=0,
\qquad
\left(-\tfrac12-a^2+b+c\right)^2-\tfrac14-a^2-c=0.
\]
The first equation gives $b=c^2-2c$, and substitution into the second yields
\[
(a-c)(a+c)(a^2-c^2+2c)=0.
\]
Therefore
\[
(b,c)=(a^2-2a,a),\qquad
(b,c)=(a^2+2a,-a),
\]
or
\[
(b,c)=(a^2,1+\sqrt{a^2+1}),
\qquad
(b,c)=(a^2,1-\sqrt{a^2+1}).
\]

For each of these four possibilities the equation is elementary to integrate;
one obtains
\[
g(z)=A e^{az}P_1(z)+B e^{-az}P_2(z),
\]
with $P_1,P_2$ as in the statement. The conditions $g(0)\neq0$ and
$g(1)\neq0$ exclude precisely the choices for which $g$ and
$P(z)=z(z-1)$ have a common zero. Finally, Theorem~\ref{thm:razloz_v_sum_1} gives $g(z)\to\infty$ in every
Stokes sector; hence neither exponential contribution may vanish. Thus
$A,B\neq0$, and the required decomposition follows.
\end{proof}

\subsection{Pullback constructions for general $P$}

The preceding examples come from explicit model equations. The following
construction pulls such equations back by a polynomial primitive of $P$, and
therefore gives examples with an arbitrary prescribed zero polynomial.

Let $\mathcal P$ be a polynomial primitive of $P$, that is, $\mathcal P'(z)=P(z)$.

\begin{theorem}\label{thm:pullback-construction}
Let $P$ and $T$ be polynomials of degrees $p\in\mathbb N_0$ and
$t\in\mathbb N_0$, respectively, and let $h$ be an entire solution of
\[
h''+Th=0
\]
which tends to $\infty$ in every Stokes sector of this equation. Let
\[
H(z)=a\mathcal P(z)+b,
\qquad a\in\mathbb C\setminus\{0\},\quad b\in\mathbb C,
\]
and set
\[
g=h\circ H.
\]
Assume that $g$ and $P$ have no common zeros. Then
\[
\frac{P(z)}{g^2(z)}
=
\sum_{n=1}^{\infty}\frac{c_n}{(z-t_n)^2},
\qquad
\sum_{n=1}^{\infty}\frac{|c_n|}{|t_n|^2}<\infty,
\]
where $\{t_n\}$ is the zero set of $g$. Moreover, $g$ is an entire
function of order $(p+1)\left(1+\frac{t}{2}\right)$.

\end{theorem}

\begin{proof}
We have
\[
g'=(h'\circ H)H',
\qquad
g''=(h''\circ H)H'^2+(h'\circ H)H''.
\]
Using $h''=-Th$, we obtain
\[
g''=(h'\circ H)H''-(T\circ H)H'^2g.
\]
Therefore
\[
\frac{P'}{P}g'-g''
=
\left(\frac{P'}{P}H'-H''\right)(h'\circ H)
+
(T\circ H)H'^2g.
\]
But $H'=aP$ and $H''=aP'$, so the first term vanishes. Hence
\[
\frac{P'}{P}g'-g''
=
a^2P^2(T\circ H)g,
\]
and so $g$ satisfy equation \eqref{eq:osnov_eq_for_examples}  with $Q=a^2P^3(T\circ H)$ and $\deg Q=3p+t(p+1)$. It follows that
\[
\rho(g)=\frac{\deg Q-\deg P+2}{2}
=
(p+1)\left(1+\frac{t}{2}\right).
\]
Since
\[
H(z)=\frac{a}{p+1}z^{p+1}(1+o(1)),
\qquad z\to\infty,
\]
the preimages under $H$ of the Stokes sectors of $h''+Th=0$ are asymptotically
the Stokes sectors of the pulled-back equation. By assumption, $h$ tends to
$\infty$ in every such sector, and hence so does $g=h\circ H$. Since
$g$ and $P$ have no common zeros, Theorem~\ref{thm:razloz_v_sum_1} gives
the asserted decomposition.
\end{proof}

\begin{remark}
Theorem~\ref{thm:pullback-construction} provides many explicit examples. In
particular, it applies to
\[
\frac{z(z-1)}
{\sin^2\!\left(\frac{z^3}{3}-\frac{z^2}{2}+1\right)}
\]
and
\[
\frac{(z^2+z+1)(z-2)}
{\Bi^2\!\left(\frac{z^4}{4}-\frac{z^3}{3}-\frac{z^2}{2}-2z\right)}.
\]
\end{remark}

\medskip 

\section{Further questions}

We conclude with two questions which lie outside the finite-order theory
developed above.

\begin{question}
Do there exist functions of infinite order of the form
\[
f(z)=\sum_{n=1}^{\infty}\frac{c_n}{(z-t_n)^2},
\qquad
\sum_{n=1}^{\infty}\frac{|c_n|}{|t_n|^2}<\infty,
\]
with only finitely many zeros?
\end{question}

The reduction used in this paper gives, without the finite-order assumption,
an equation
\[
Pg''-P'g'+Qg=0
\]
with $Q$ entire. In the finite-order case $Q$ is a polynomial, and the
problem becomes accessible through the asymptotic theory of second-order
equations with rational coefficients. If $Q$ is transcendental, very little
seems to be known in the required generality about the zero distribution of
entire solutions. This is the main obstacle to extending the present
results to infinite order.

\begin{question}
Can one characterize the order-one examples for which the logarithmic
localization is sharp?
\end{question}

In the borderline case $\rho=1$, the general theorem gives only logarithmic
neighborhoods of the Stokes rays. In special cases, such as the elementary
trigonometric examples, the poles may lie on the rays themselves. It would be
interesting to distinguish these cases from those in which the distance from
the poles to the corresponding Stokes rays tends to infinity, while remaining
of logarithmic size.

\end{document}